\documentclass{siamart171218}

\usepackage{tikz}
\usepackage{pgf}
\usetikzlibrary{arrows}
\usetikzlibrary{calc}
\usetikzlibrary{decorations.pathmorphing}

\usepackage{rotating} 

\usepackage{mathrsfs}
\DeclareMathAlphabet{\mathpzc}{OT1}{pzc}{m}{it}

\usepackage{color}
\usepackage{amsmath}
\usepackage{graphicx}
\usepackage{amsfonts}
\usepackage{amssymb}%
\usepackage{latexsym}
\usepackage{psfrag}
\usepackage{accents}

\newtheorem{remark}[theorem]{Remark}
\newtheorem{example}[theorem]{Example}

\newcommand{\la}{\lambda}

\newcommand{\R}{\mathbb{R}}

\newcommand{\cH}{{\cal H}}

\newcommand{\cR}{{\cal R}}

\newcommand{\cD}{{\cal D}}

\newcommand{\cO}{{\cal O}}

\newcommand{\bx}{x}

\newcommand{\bxi}{\mathbf{\xi}}

\newcommand{\by}{y}

\newcommand{\re}{{\rm e}}
\newcommand{\ri}{{\rm i}}
\newcommand{\rd}{{\rm d}}

\newcommand{\beq}{\begin{equation}}
\newcommand{\eeq}{\end{equation}}
\newcommand{\beqs}{\begin{equation*}}
\newcommand{\eeqs}{\end{equation*}}
\newcommand{\bit}{\begin{itemize}}
\newcommand{\eit}{\end{itemize}}
\newcommand{\ben}{\begin{enumerate}}
\newcommand{\een}{\end{enumerate}}
\newcommand{\bal}{\begin{align}}
\newcommand{\eal}{\end{align}}
\newcommand{\bals}{\begin{align*}}
\newcommand{\eals}{\end{align*}}
\newcommand{\bse}{\begin{subequations}}
\newcommand{\ese}{\end{subequations}}
\newcommand{\bpr}{\begin{proposition}}
\newcommand{\epr}{\end{proposition}}
\newcommand{\bre}{\begin{remark}}
\newcommand{\ere}{\end{remark}}
\newcommand{\bpf}{\begin{proof}}
\newcommand{\epf}{\end{proof}}
\newcommand{\ble}{\begin{lemma}}
\newcommand{\ele}{\end{lemma}}
\newcommand{\bco}{\begin{corollary}}
\newcommand{\eco}{\end{corollary}}
\newcommand{\bex}{\begin{example}}
\newcommand{\eex}{\end{example}}
\newcommand{\bth}{\begin{theorem}}
\newcommand{\enth}{\end{theorem}}

\newcommand{\Rea}{\mathbb{R}}
\newcommand{\Com}{\mathbb{C}}

\newcommand{\Oi}{{\Omega_-}}

\newcommand{\eps}{\varepsilon}

\newcommand{\pdiff}[2]{\frac{\partial #1}{\partial #2}}

\newcommand{\gu}{\nabla u}

\newcommand{\half}{\frac{1}{2}}

\newcommand{\LtG}{{L^2(\bound)}}

\newcommand{\LtGt}{{\LtG\rightarrow \LtG}}

\newcommand{\tendi}{\rightarrow \infty}
\newcommand{\tendo}{\rightarrow 0}

\newcommand{\opA}{A'_{k,\eta}}
\newcommand{\opABW}{A_{k,\eta}}
\newcommand{\opAinv}{(A'_{k,\eta})^{-1}}
\newcommand{\opABWinv}{A^{-1}_{k,\eta}}

\newcommand{\normAinv}{\|\opAinv\|}
\newcommand{\normABWinv}{\|\opABWinv\|}

\def\XXint#1#2#3{{\setbox0=\hbox{$#1{#2#3}{\int}$}
     \vcenter{\hbox{$#2#3$}}\kern-.5\wd0}}

\usepackage{hyperref}
\definecolor{myblue}{rgb}{0,0,0.6}
\hypersetup{colorlinks=true,
linkcolor=myblue,citecolor=myblue,filecolor=myblue,urlcolor=myblue}
\newcommand*{\N}[1]{\left\|#1\right\|}

\allowdisplaybreaks[4]

\newcommand{\tfa}{\text{ for all }}
\newcommand{\tfor}{\text{ for }}

\newcommand{\tin}{\text{ in }}
\newcommand{\ton}{\text{ on }}
\newcommand{\tas}{\text{ as }}
\newcommand{\tand}{\text{ and }}
\newcommand{\tst}{\text{ such that }}

\newcommand{\vertiii}[1]{{\left\vert\kern-0.25ex\left\vert\kern-0.25ex\left\vert #1
    \right\vert\kern-0.25ex\right\vert\kern-0.25ex\right\vert}}

\newcommand{\bound}{{\partial \obstacle_+}}

\newcommand{\cutoff}{\chi R(k)\chi}
\newcommand{\LtLt}{L^2(\obstacle_+)\rightarrow L^2(\obstacle_+)}

\usepackage{soul}

\definecolor{jwcol}{rgb}{0.8,0,0}

\definecolor{dalcol}{rgb}{0,0.8,0}

\definecolor{escol}{rgb}{0,0,0.8}
\definecolor{estcol}{rgb}{0,0.5,0}
\definecolor{esnewcol}{rgb}{0,0.5,0}

\newcommand{\obstacle}{{\cO}}
\newcommand{\Real}{{\rm Re}}
\newcommand{\Imag}{{\rm Im}}
\newcommand{\Res}{{\rm Res} }
\newcommand{\Cw}{C_{\rm w}}

\newcommand{\hFEM}{{h_{\rm FEM}}}

\newcommand{\uin}{{u_{\rm in}}}
\newcommand{\uout}{{u_{\rm out}}}

\newcommand{\supp}{{\rm supp}}

\begin{document}

\title{For most frequencies, strong trapping has a weak effect in frequency-domain scattering}

\author{D.~Lafontaine\footnotemark[1]\,\,, E.~A.~Spence\footnotemark[2]\,\,, J.~Wunsch\footnotemark[3]}

\date{\today}

\footnotetext[1]{Department of Mathematical Sciences, University of Bath, Bath, BA2 7AY, UK, \tt D.Lafontaine@bath.ac.uk }
\footnotetext[2]{Department of Mathematical Sciences, University of Bath, Bath, BA2 7AY, UK, \tt E.A.Spence@bath.ac.uk }
\footnotetext[3]{Department of Mathematics, Northwestern University, 2033 Sheridan Road, Evanston IL 60208-2730, US, \tt jwunsch@math.northwestern.edu}

\maketitle

\begin{abstract}
It is well known that when the geometry and/or coefficients allow
stable trapped rays, the 
outgoing solution operator of the Helmholtz equation 
grows \emph{exponentially} through a sequence of real frequencies tending to infinity.

In this paper we show that, even in the presence of the strongest-possible trapping, if a set of frequencies of arbitrarily small measure is excluded, the Helmholtz solution operator grows at most \emph{polynomially} as the frequency tends to infinity. 

One significant application of this result is in the convergence analysis of several numerical methods for solving the Helmholtz equation at high frequency that are based on a polynomial-growth assumption on the solution operator  (e.g. $hp$-finite elements, $hp$-boundary elements, certain multiscale methods). The result of this paper shows that this assumption holds, even in the presence of the strongest-possible trapping, for most frequencies. 

\

\textbf{Keywords.}
Helmholtz equation, high frequency, trapping, resolvent, scattering theory, resonance, finite element method, boundary element method.
\end{abstract}

\section{Introduction}\label{sec:intro}

\subsection{Motivation: bounds on the solution operator under trapping}\label{sec:1.1}

Trapping and nontrapping are central concepts in scattering
theory. This paper is concerned with the behaviour of the
outgoing solution
operator in frequency-domain scattering problems (a.k.a.~the
resolvent) in the presence of strong trapping. Our results hold for a
wide variety of boundary-value problems where the differential
operator is the Helmholtz operator $\Delta + k^2$ outside some compact
set; indeed, we work in the framework of \emph{black-box scattering}
introduced by Sj\"ostrand--Zworski in \cite{SjZw:91} and recalled
briefly in \S\ref{sec:blackbox}. For simplicity, in this introduction
we focus on the exterior Dirichlet problem (EDP) for the Helmholtz
equation; i.e.~the problem of, given a bounded, open set
$\obstacle_-\subset\Rea^n$, $n\geq 2$, such that the open complement
$\obstacle_+:=\Rea^n\setminus\overline{\obstacle_-}$ is connected and
$\partial \obstacle_+$ is Lipschitz, $f\in L^2(\obstacle_+)$ with
compact support, and frequency $k>0$, finding
$u\in H^1_{\rm loc}(\obstacle_+)$ such that \beq\label{eq:edp} \Delta
u +k^2 u =-f \quad\tin \obstacle_+, \quad \gamma u =0 \ton \partial
\obstacle_+, \eeq (where $\gamma$ denotes the trace operator on
$\partial \obstacle_+$) and
\begin{equation}\label{eq:src}
\pdiff{u}{r}(x) - \ri k u(x) = o\left(\frac{1}{r^{(d-1)/2}}\right),
\end{equation}
as $r \tendi$, uniformly in $\widehat{x}:= x/r$ (with this last condition the \emph{Sommerfeld radiation condition}, 
and solutions satisfying this condition known as \emph{outgoing}).
A classic result of Rellich (see, e.g.~\cite[Theorems 3.33 and 4.17]{DyZw:19}) implies that the solution of the EDP is unique for all $k$.
Formulating the EDP as a variational problem in a large ball as in \S\ref{sec:FEMqo} below, one can then apply
the Fredholm alternative (see, e.g., \cite[Theorem 2.6.6]{Ne:01})
to obtain that the solution exists for all $k$
and, given $R>0$ such that $\supp f\subset B_R:= \{ x : |x| <R\}$ and
$k_0>0$, 
\beq\label{eq:res1}
\N{\gu}_{L^2(\obstacle_+ \cap B_R)} + k \N{u}_{L^2(\obstacle_+ \cap B_R)} \leq \Upsilon(k, \obstacle_-, R, k_0) \N{f}_{L^2(\obstacle_+)}
\eeq
for all $k\geq k_0$, where $\Upsilon(k, \obstacle_-, R, k_0)$ is some (a priori unknown) function of $k, \obstacle_-, R$, and $k_0$.

It is convenient to write bounds such as \eqref{eq:res1} in terms of the 
outgoing cut-off resolvent $\chi R(k)\chi : L^2(\obstacle_+) \rightarrow H^1(\obstacle_+)$ for $k\in \Rea\setminus\{0\}$,
where $\chi \in C^\infty_{\rm comp}(\overline{\obstacle_+})$
and $$R(k):= -(\Delta +k^2)^{-1},$$
on $\partial \obstacle_+$, is defined by analytic continuation from
$R(k) : L^2(\obstacle_+)\rightarrow L^2(\obstacle_+)$ for $\Im k>0$
(this definition impiles that the radiation condition
\eqref{eq:src} is satisfied for $k\in
\Rea\setminus\{0\}$); see, e.g., \cite[\S3.6, Theorem 4.4, and Example 2 on Page 229]{DyZw:19}.
The bound \eqref{eq:res1} then becomes
\begin{align}\label{eq:res3}
\N{\chi R(k)\chi }_{L^2(\obstacle_+)\rightarrow L^2(\obstacle_+)} &\leq \frac{\Upsilon(k, \obstacle_-, \chi, k_0)}{k}, \\
\N{\chi R(k)\chi}_{L^2(\obstacle_+)\rightarrow H^1 (\obstacle_+)} &\leq \frac{\Upsilon(k, \obstacle_-, \chi, k_0)}{\min(k_0,1)},
\nonumber
\end{align}
for all $k\geq k_0$.
Having obtained an $L^2\rightarrow L^2$ bound on $\chi R(k)\chi$, an
$L^2\rightarrow H^1$ bound can be obtained from Green's identity
(i.e.~multiplying the PDE in \eqref{eq:edp} by $\overline{u}$ and
integrating by parts; see, e.g., \cite[Lemma 2.2]{Sp:14}) and so we
focus on $L^2\rightarrow L^2$ bounds from now on.  
The Schwartz kernel of the outgoing resolvent, often referred to as
  the outgoing Green function, is necessarily singular at the
  diagonal, so it is $L^2$ mapping estimates that seem most natural in
  this context.

When $\obstacle_+$ has $C^\infty$ boundary and is \emph{nontrapping},
i.e.~all billiard trajectories starting in an exterior neighbourhood
of $\obstacle_-$ escape from that neighbourhood after some uniform
time, one can show that $\Upsilon$ in \eqref{eq:res3} is independent
of $k$, i.e.  given $k_0>0$, \beq\label{eq:nt_estimate}
\N{\cutoff}_{\LtLt}\lesssim \frac{1}{k} \quad\tfa k\geq k_0, \eeq
where the notation $a\lesssim b$ means that there exists a $C>0$,
independent of $k$ (but dependent on $k_0$, $\obstacle_+$, and
$\chi$), such that $a\leq C b$.  This classic nontrapping resolvent
estimate was first obtained by the combination of the results on
propagation of singularities for the wave equation on manifolds with
boundary by Andersson--Melrose \cite{AnMe:77}, Melrose \cite{Me:75},
Taylor \cite{Ta:76}, and Melrose--Sj\"ostrand \cite{MeSj:78, MeSj:82} with
either the parametrix method of Vainberg \cite{Va:75} (see
\cite{Ra:79}) or the methods of Lax--Phillips \cite{LaPh:89} (see
\cite{Me:79}). (See \cite{GaSpWu:18} for precise estimates on the
omitted constant in the inequality \eqref{eq:nt_estimate}.)

On the other hand, when $\obstacle_+$ is \emph{trapping}, a loss is
unavoidable in the cut-off resolvent; indeed, at least in the
analogous case of
semiclassical scattering by a potential, if trapping exists then one
has a semiclassical lower bound by \cite[Th\'eor\`eme 2]{BoBuRa:10} 
(see also \cite[Theorem 7.1]{DyZw:19}),
which in our notation 
implies that there exists a sequence of frequencies $0<k_1<k_2<\ldots$, with $k_j\tendi$, such that
  \beq\label{eq:lower_trapp} 
\|\chi R(k_j)\chi\|_{L^2 \rightarrow L^2}\gtrsim
\frac{\log (2+k_j)}{k_j},  \quad j=1,2,\ldots,
\eeq
and one expects the strength of the loss to depend on the strength of the
trapping. In the standard example of \emph{hyperbolic trapping},
when $\obstacle_-$ equals the union of two disjoint convex
obstacles with strictly positive curvature (see Figure
\ref{fig:examples}(a), the lower bound \eqref{eq:lower_trapp} is
achieved, since \beqs
\N{\cutoff}_{\LtLt}\lesssim \frac{\log (2+k)}{k} \quad\tfa k\geq k_0,
\eeqs by \cite[Proposition 4.4]{Bu:04} (which is based on now classic
work of Ikawa \cite{Ik:88}).  In the standard example of \emph{parabolic
  trapping}, when $\obstacle_-$
equals the union of two disjoint, aligned squares, in 2-d, or cubes,
in 3-d, (see Figure \ref{fig:examples}(b)), the cut-off resolvent
suffers a polynomial loss over the nontrapping estimate, with the
bound \beqs
\N{\cutoff}_{\LtLt}\lesssim k \quad\tfa k\geq k_0 \eeqs proved in
\cite[Theorem 1.9]{ChSpGiSm:17}; variable-power polynomial losses have
also been exhibited in \cite[Theorem 2]{ChWu:13} in cases of
degenerate-hyperbolic trapping in the setting of scattering by
metrics.

For general $\obstacle_+$  with  $C^\infty$ boundary, the cut-off resolvent can grow at most exponentially in $k$ by the bound of Burq \cite[Theorem 2]{Bu:98}
\beqs
\N{\cutoff}_{L^2(\obstacle_+) \rightarrow L^2(\obstacle_+)}\lesssim \re^{\alpha k} \quad\tfa k\geq k_0
\eeqs
for some $\alpha= \alpha(\obstacle_-, k_0)>0$. In the presence of the
strongest possible trapping -- so called \emph{elliptic trapping} -- this exponential growth of the cut-off resolvent is achieved. Indeed, if $\obstacle_-$ has an ellipse-shaped cavity (see Figure \ref{fig:examples}(c)) then there exists a sequence of frequencies $0<k_1<k_2<\ldots$, with $k_j\tendi$, and $\alpha>0$ such that
\beq\label{eq:ellipse}
\N{\chi R(k_j)\chi}_{\LtLt}\gtrsim \re^{\alpha k_j}, \quad j=1,2,\ldots,
\eeq
see, e.g., \cite[\S2.5]{BeChGrLaLi:11}.  More generally, if  there exists an elliptic trapped ray (i.e.~an elliptic closed broken geodesic),
and $\partial \obstacle_+$ is analytic in neighbourhoods of the vertices of the broken geodesic, then the resolvent can grow at least as fast as $\exp{(\alpha k_j^q)}$, through a sequence $k_j$ as above and for some range of $q\in(0,1)$, by the quasimode construction of Cardoso--Popov \cite{CaPo:02} (note that Popov proved \emph{superalgebraic} growth for certain elliptic trapped rays when $\partial \Oi$ is smooth in \cite{Po:91}).

The question this paper answers is \emph{how does the cut-off resolvent behave under elliptic trapping when $k$ is not equal to one of the ``bad" frequencies $k_j$?}

Our answer to this question uses the fact that the growth  \eqref{eq:ellipse} of the cut-off resolvent through the real sequence $k_j$ under trapping is due to the presence of (complex) resonances lying in the lower-half complex $k$-plane, close to the real axis. The ``bad" real frequencies $k_j$ then correspond to the real parts of these (complex) resonances. The strength of the trapping and how close the resonances are to the real axis are intimately related. Indeed, in elliptic trapping, the resonances are super-algebraically close to the real axis, causing at least superalgebraic growth of the cut-off resolvent, whereas in hyperbolic trapping the resonances stay a fixed distance away from the real axis, hence the weak logarithmic loss over the nontrapping resolvent estimate; see the recent overview discussion in \cite[\S2.4]{Zw:17} and the references therein.

\begin{figure}
\begin{center}
\includegraphics[width=0.75\textwidth]{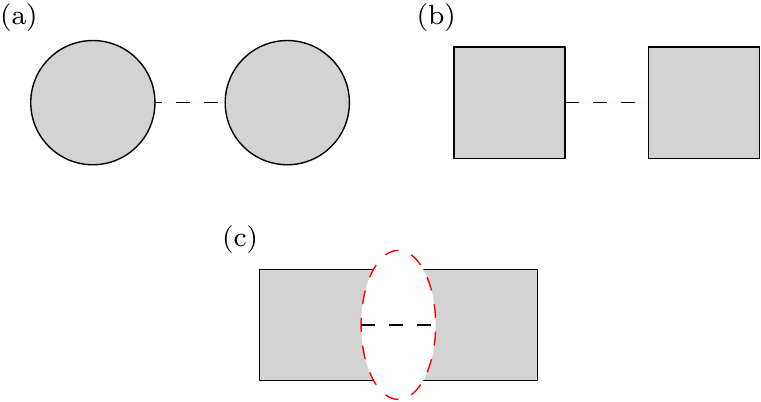}
\caption{Examples of (a) hyperbolic trapping, (b) parabolic trapping, and (c) elliptic trapping, with a trapped ray for each indicated by a black dashed line.}
\label{fig:examples}
\end{center}
\end{figure}

\subsection{Statement of main results (in the setting of impenetrable-Dirichlet-obstacle scattering)}\label{sec:main}

In the setting of scattering by an impenetrable Dirichlet obstacle our main result is the following. This result is valid (and hence stated) for all Lipschitz obstacles, but is of primary interest when the obstacle contains an elliptic trapped ray.

\begin{theorem}[Polynomial resolvent estimate for most frequencies]\label{thm:intro1}
Let $\obstacle_-\subset\Rea^n$, $n\geq 2$,  be a bounded open set such that the open complement $\obstacle_+:=\Rea^n\setminus\overline{\obstacle_-}$ is connected and $\partial\obstacle_+$ is Lipschitz. Let $R(k)$ be defined as in \S\ref{sec:1.1}.
Then, given $k_0>0$, $\delta>0$, and $\eps>0$, there exists $C=C(k_0, \delta, \eps, n)>0$ and a set $J\subset [k_0,\infty)$ with $|J|\leq \delta$ such that
\beq\label{eq:intro1}
\N{\chi R(k)\chi}_{\LtLt}\leq Ck^{5n/2+ \eps}
\quad \tfa k \in [k_0,\infty)\backslash J.
\eeq
\end{theorem}

In other words, even in the presence of elliptic trapping, outside an
arbitrary-small set of frequencies, the resolvent is always
polynomially bounded, with an exponent depending only on the
dimension. We make the following remarks.
\ben
\item The analogue of Theorem \ref{thm:intro1} in the black-box-scattering framework is given as Theorem \ref{thm:bb1} below -- a resolvent estimate identical to \eqref{eq:intro1} in its $k$-dependence is therefore valid in a wide range of settings, including scattering by an impenetrable Neumann obstacle, by a penetrable obstacle, by a potential, by elliptic and compactly-supported perturbations of Laplacian, and on finite volume surfaces (see \S\ref{sec:blackbox} and the references therein).
\item 
The proof of Theorem \ref{thm:intro1} uses 
(i) a polynomial bound on the density of resonances (\eqref{eq:count} below), (ii) a bound on the resolvent away from resonances (Theorem \ref{thm:TZkey} below) and (ii) the semiclassical maximum principle (Theorem \ref{thm:scmp} below). 
The bounds in (i) were originally pioneered by Melrose, and then further developed by Sj\"ostrand, Sj\"ostrand--Zworski, Vodev, and Zworski (see the references below \eqref{eq:count}, and also the literature overviews in  \cite[\S\S3.13 and 4.7]{DyZw:19}). 
The results in (ii) and (iii) are due to Tang--Zworski \cite{TaZw:00}. We highlight that, in fact,
\cite[Proposition 4.6]{TaZw:00} notes that the cut-off resolvent is bounded polynomially in regions of the complex plane that include intervals of the real axis away from resonances; the difference here is that we seek to control the measure of these intervals.
\item 
When we have finer information about the distribution of resonances, we can lower
the exponent in \eqref{eq:intro1} and also obtain a bound on the
measure of the set $\big\{k :\Vert \chi R(k) \chi\Vert_{L^2\rightarrow
  L^2}>\la^{s}\big\}\cap [\la, \la+1)$; see Theorem \ref{thm:bb2}. In particular, for scattering by a obstacle with $C^{1,\sigma}$ boundary (for some $0<\sigma<1$),
known results on Weyl laws for resonances \cite{PeZw:99} allow us to improve the exponent in \eqref{eq:intro1} to
$5n/2-1+\eps$ for all $\eps>0$; see Corollary~\ref{cor:bb3}.
Another scenario where we have an improvement in the exponent is that of
scattering
by a smooth, strictly convex, penetrable obstacle; see Corollary \ref{cor:transmission}.

\item We do not know the sharp value of the exponent in the bound
  \eqref{eq:intro1}. Under a hypothesis that there exist quasimode
  solutions to the equation (often easy to construct in strong trapping
  situations) whose frequencies are well distributed, we obtain a lower bound for all frequencies  of
  $\|\chi R(k)\chi\|_{L^2\rightarrow L^2} \gtrsim k^{n-2}$: see
  Lemma \ref{lem:lowerbd} below.

\item 
Similar results to Theorem \ref{thm:intro1} about relatively ``good" behaviour of the Helmholtz solution operator under elliptic trapping as long as $k$ is outside some finite set
were proved by Capdeboscq and co-workers for scattering by a penetrable ball in \cite[Theorem 6.5]{Ca:12} for 2-d and \cite[Theorem 2.5]{CaLePa:12} for 3-d. 
These results use the explicit expression for the solution in terms of an expansion in Fourier series (2-d) or spherical harmonics (3-d), with coefficients given by Hankel and Bessel functions, to
bound the scattered field outside the obstacle in terms of the incident field, with a loss of derivatives (corresponding to a loss of powers of $k$). 
At least when the contrast in wave speeds inside and outside the obstacle is sufficiently large, \cite[Lemma 6.2]{Ca:12} and \cite[Lemma 3.6]{CaLePa:12} show that the scattered field everywhere outside the obstacle is polynomially bounded in $k$ for $k$ outside a set of small, finite measure; see Remark \ref{rem:Cap} below for more discussion fn the results of \cite{Ca:12, CaLePa:12}.

\item As noted in \S\ref{sec:1.1}, when the obstacle $\obstacle_-$
  contains an ellipse-shaped cavity, the resolvent grows exponentially
  through a sequence $k_j$ \eqref{eq:ellipse}; in this situation
  Theorem \ref{thm:intro1} implicitly contains information about the
  widths of the peaks in the norm of the resolvent at $k_j$. We are
  not aware of any results in the literature about the widths of these
  peaks in the setting of obstacle scattering, but precise information
  about the widths and heights of peaks in the transmission
  coefficient for model resonance problems in one space dimension can
  be found in \cite{ShWe:13}, \cite{AbSh:16}.
\item Complementary results (in a different direction to Theorem \ref{thm:intro1}) about ``good" behaviour of the resolvent in trapping scenarios can be found in in \cite[Theorem 1.1]{CaVo:02}, \cite[Theorem 4]{Bu:02a}, and  \cite[Theorems 1.1, 1.2]{DaVa:13}.
Indeed, \cite[Theorem 1.1]{CaVo:02} proves that, even in the presence of trapping, the nontrapping resolvent estimate \eqref{eq:nt_estimate} holds when the support of $\chi$ is sufficiently far away from the obstacle (\cite[Theorem 4]{Bu:02a} proves this result up to factors of $\log k$). The results \cite[Theorems 1.1, 1.2]{DaVa:13} prove the analogue of this result in the setting of scattering by a potential and/or by a metric when the cut-off functions are replaced by semiclassical pseudodifferential operators restricting attention to areas of phase space isolated from the trapped set.
\item A result similar in spirit to Theorem \ref{thm:intro1} in the case of bounded domains and eigenfunctions is \cite[Theorem 1]{HiMa:12};
this result obtains an improvement on previous bounds about concentration of eigenfunctions for frequencies outside a specific set (corresponding to eigenvalues of a subdomain).
\een

Using the results of \cite{BaSpWu:16} (a sharpening of previous arguments in \cite{LaVa:11,Sp:14}, and written down in  \cite[Lemma 4.3]{ChSpGiSm:17} for a resolvent estimate with arbitrary $k$-dependence), the resolvent estimate \eqref{eq:intro1} 
immediately implies bounds on the Dirichlet-to-Neumann (DtN) map described in the following corollary. 

To state these bounds we first recall the definition of the weighted $H^1$ norm: $\|v\|^2_{H^1_k(D)} := \|\nabla v\|^2_{L^2(D)} + k^2 \|v\|^2_{L^2(D)}$ for $D$ an open set. We use this definition below, both with $D= \obstacle_+$ and with $D=\partial \obstacle_+$; in the latter case the gradient is understood as the surface gradient on $\partial \obstacle_+$; see, e.g., \cite[pp.~98--99]{Mc:00}. The weighted Sobolev spaces $H^s_k(\partial \obstacle_+)$ for $s\in (0,1)$ are then defined by, e.g., \cite[Chapter 3]{Mc:00}, with the norms defined by interpolation; see, e.g., \cite[\S2.3]{ChSpGiSm:17} and \cite{ChHeMo:15}. Finally, let $\partial_\nu$ denote the normal-derivative operator defined by, e.g., \cite[Lemma 4.3]{Mc:00} (recall that this operator is such that, when $v \in H^2_{\rm loc}(\obstacle_{+})$, $\partial_{\nu}v = \nu \cdot \gamma \nabla v$).

\begin{corollary}[Bounds on the DtN map for most frequencies] \label{cor:DtN} 
Let $\obstacle_+$ be as in Theorem \ref{thm:intro1}.
  Let $u\in H^1_{\mathrm{loc}}(\obstacle_+)$ be a solution to the Helmholtz equation $\Delta u + k^2 u = 0$ in $\obstacle_+$ that satisfies the Sommerfeld radiation condition \eqref{eq:src} and the boundary condition $\gamma u=g$. 
Then, given $\chi \in C^\infty_{\rm comp}(\overline{\obstacle_+})$, 
$k_0>0$, $\delta>0$, and $\eps>0$, there exists $C'=C'(k_0, \delta, \eps, n, \obstacle_-, \chi)>0$ and a set $J\subset [k_0,\infty)$ with $|J|\leq \delta$ such that
\beqs
\|\chi u\|_{H^1_k(\obstacle_+)} + \|\partial_{\nu} u\|_{L^2(\partial\obstacle_+)} \leq C'k^{5n/2+ \eps}\|g\|_{H_k^1(\partial\obstacle_+)}
\quad \tfa k \in [k_0,\infty)\backslash J,
\eeqs
if $g\in H^1(\partial\obstacle_+)$. Furthermore, uniformly for $0\leq s \leq 1$, and provided $g\in H^{s}(\partial\obstacle_+)$,
\beqs
\|\partial_{\nu} u\|_{H_k^{s-1}(\partial\obstacle_+)} \lesssim k^{5n/2+ \eps} \|g\|_{H_k^{s}(\partial\obstacle_+)} \,\, \mbox{and} \,\, \|\partial_{\nu} u\|_{H^{s-1}(\partial\obstacle_+)} \lesssim k^{5n/2+1+ \eps} \|g\|_{H^{s}(\partial\obstacle_+)}
\eeqs
for all $k\in [k_0,\infty)\backslash J.$
\end{corollary}

\subsection{Applications to numerical analysis of Helmholtz scattering problems}

\subsubsection{The use of bounds on the resolvent in numerical analysis}\label{sec:NA1}

The Helmholtz equation is arguably the simplest-possible model of wave propagation, and therefore there has been considerable research into designing accurate and efficient methods for solving it numerically, especially when the frequency is large and the solution is highly oscillatory.
A bound on the solution operator for a boundary-value problem underpins the numerical analysis of any numerical method for solving that particular problem; consequently, the non-trapping resolvent estimate \eqref{eq:nt_estimate} for the Helmholtz equation 
has been widely used by the numerical-analysis community in the frequency-explicit analysis of numerical methods for Helmholtz problems. 

The following is a non-exhaustive list of papers 
on the frequency-explicit convergence analysis of numerical methods for solving the Helmholtz equation where a central role is played by \emph{either}
 the non-trapping resolvent estimate \eqref{eq:nt_estimate}, \emph{or} its analogue (with the same $k$-dependence) for the commonly-used approximation of the exterior problem where the exterior domain $\obstacle_+$ is truncated and an impedance boundary condition is imposed:
\bit
\item conforming FEMs (including continuous interior-penalty methods) \cite[Proposition 2.1]{MaIhBa:96}, \cite[Proposition 8.1.4]{Me:95}, 
\cite[Theorem 2.2]{ShWa:05}, \cite[Theorem 3.1]{ShWa:07},
\cite[Lemma 2.1]{HaHu:08},
\cite[Lemma 3.5]{MeSa:10}, \cite[Assumptions 4.8 and 4.18]{MeSa:11}, \cite[\S2.1]{EsMe:12},
\cite[Theorem 3.1]{Wu:13}, \cite[\S3.1]{ZhWu:13}, \cite[\S3.2.1]{EsMe:14},
\cite[Remark 3.2]{DuWu:15}, \cite[Remark 3.1]{DuZh:16},
\cite[Assumption 1]{ChNi:18}, \cite[Definition 2]{ChNi:18a}, \cite[Theorem 3.2]{GrSa:18}, \cite[Lemma 6.7]{GaSpWu:18}, \cite[Eq.~4]{BuNeOk:18}, \cite[Eq.~1.20]{LaSpWu:19a},

\item 
least squares methods \cite[Assumption A1]{ChQi:17}, \cite[Remark 1.2]{BeMe:18}, \cite[Assumption 1 and equation after Eq.~5.37]{HuSo:19},
\item DG methods based on piece-wise polynomials 
\cite[Theorem 2.2]{FeWu:09}, \cite[Theorem 2.1]{FeWu:11}, 
\cite[Assumption 3]{DeGoMuZi:12},
\cite[\S3]{FeXi:13}, \cite[Assumption A (Eq.~4.5)]{HoSh:13}, \cite[Eq.~4.4]{MePaSa:13}, 
\cite[Remark 3.2]{CuZh:13}, \cite[Eq.~2.4]{ChLuXu:13},
\cite[Eq.~4.3]{MuWaYe:14},
 \cite[Remark 3.1]{ZhDu:15}, \cite[Theorem 2.2]{SaZe:15},
\item plane-wave/Trefftz-DG methods \cite[Theorem 1]{AmDjFa:09}, \cite[Eq.~3.5]{HiMoPe:11}, \cite[Theorem 2.2]{HiMoPe:14}, \cite[Lemma 4.1]{AmChDiDjFi:14}, \cite[Proposition 2.1]{HiMoPe:16}, 
\item multiscale finite-element methods 
\cite[Eq.~2.3]{GaPe:15}, 
\cite[\S1.2]{BrGaPe:15}, 
\cite[Assumption 5.3]{Pe:17}, 
\cite[Theorem 1]{BaChGo:17}, 
\cite[Assumption 3.8]{OhVe:18}, 
\cite[Assumption 1]{ChVa:18},
\item integral-equation methods 
\cite[Eq.~3.24]{LoMe:11}, \cite[Eq.~4.4]{Me:12}, \cite[Chapter 5]{ChGrLaSp:12}, 
\cite[Theorem 3.2]{GrLoMeSp:15}, \cite[Remark 7.5]{ZhBa:15}, 
\cite[Theorem 2]{EpGrHa:16},
\cite[Theorem 3.2]{GaMuSp:17},
\cite[Assumption 3.2]{GiChLaMo:19},
\eit
In addition, the following papers focus on proving bounds on the solution of Helmholtz boundary-value problems (with these bounds often called ``stability estimates") motivated by applications in numerical analysis:
\cite{CuFe:06},
\cite{He:07},
\cite{ChMo:08},
\cite{BeChGrLaLi:11},
\cite{BaYuZh:12},
\cite{LiMaSu:13},
\cite{Sp:14},
\cite{Ch:15},
\cite{BaYu:16},
\cite{BaSpWu:16},
\cite{ChSpGiSm:17},
\cite{SaTo:18},
\cite{GrPeSp:18},
\cite{GrSa:18}
\cite{MoSp:19},
\cite{GaSpWu:18},
Of these papers, all but \cite{LiMaSu:13}, \cite{BaYu:16}, \cite{ChSpGiSm:17}, \cite{BeChGrLaLi:11} are in nontrapping situations, \cite{LiMaSu:13}, \cite{BaYu:16}, \cite{ChSpGiSm:17} are in parabolic trapping scenarios, and \cite{BeChGrLaLi:11} proves the exponential growth \eqref{eq:ellipse} under elliptic trapping. 

\subsubsection{How do numerical methods behave in the presence of trapping?}\label{sec:NA2}

We highlight three features of the behaviour of numerical methods in the presence of trapping:

First, one finds general ``bad behaviour" compared to nontrapping scenarios, independent of the frequency, because of increased number of multiple reflections. For an example of this phenomenon, see \cite[right panel of Fig.~8]{HuOp:18}, where ``bad behaviour" here means a lower compression rate of BEM matrices for trapping obstacles compared to nontrapping obstacles (and with the compression rate dependent on the strength of trapping, and worst for elliptic trapping).

Second, one finds extremely bad behaviour at real frequencies corresponding to the real parts of the (complex) resonances lying under the real axis. For example, \cite{DaDaLa:13} shows the condition number of integral-equation formulations spiking at such frequencies under parabolic trapping \cite[Fig.~18]{DaDaLa:13} and elliptic trapping \cite[Right panel of Fig.~19]{DaDaLa:13}

Third, this extremely bad behaviour at certain real frequencies is very sensitive to the frequency. For example, calculations in 
\cite[Fig.~4.7]{LoMe:11} of the norm of inverse of the integral operator $\opABW$ defined in \eqref{eq:scpo} below find that 
$\|\opABWinv\|_{L^2\rightarrow L^2} \sim 10^{11}$ at $k$ corresponding to a resonance, but changing the fifth significant figure of $k$ reduces the norm to $\sim 10^4$.
Furthermore, this sensitivity means that verifying the exponential blow-up in \eqref{eq:ellipse} is challenging. Indeed, the exponential growth of the resolvent implies 
exponential growth of $\|\opABWinv\|_{L^2\rightarrow L^2}$ (see \cite[Theorem 2.8]{BeChGrLaLi:11}, \cite[Eq.~5.39]{ChGrLaSp:12}). In the setting where the elliptic trapping is due to a ellipse-shaped cavity in the obstacle, the ``bad" frequencies correspond to certain eigenvalues of the ellipse; even knowing these eigenvalues (corresponding to the zeros of a Mathieu function; see \cite[Appendix]{BeChGrLaLi:11}) to high precision, 
\cite[\S4.8]{BeChGrLaLi:11} could only verify numerically the exponential growth of 
$\|\opABWinv\|_{L^2\rightarrow L^2}$ up to $k\approx 100$ (where the obstacle had characteristic length scale $\sim 1$).
To our knowledge, Theorem \ref{thm:intro1} is the first result rigorously describing this sensitivity of the resolvent to frequency under elliptic trapping.

\subsubsection{Three immediate applications of Theorem \ref{thm:intro1}}

The resolvent estimate in Theorem \ref{thm:intro1} can be immediately applied in all the analyses listed in \S\ref{sec:NA1} to prove results about these methods under elliptic trapping, for most frequencies.

The most exciting applications are for numerical methods whose analyses require the resolvent to be \emph{polynomially bounded in $k$, with the method depending only mildly on the degree of this polynomial}. Three such methods are
\ben
\item The $hp$-finite-element method ($hp$-FEM), where, under the assumption that the resolvent is polynomially bounded in $k$,
 the results of \cite{MeSa:10, MeSa:11, EsMe:12} establish that the finite-element method when $\hFEM\sim k^{-1}$ and $p \sim \log k$ does not suffer from the pollution effect\footnote{We use $\hFEM$ (as opposed to $h$)  to denote the  maximal element diameter in a finite element method to distinguish it from the semiclassical parameter $h=1/k$ used in \S\ref{sec:blackbox} and \S\ref{sec:proofs}.}
 ; i.e.~under this choice of $\hFEM$ and $p$, for which the total number of degrees of freedom $\sim k^n$, the method is quasi-optimal with constant independent of $k$ (see, e.g., \eqref{eq:qo} below). Similar results were then obtained for DG methods in  \cite{MePaSa:13, SaZe:15}, and for least-squares methods in \cite{ChQi:17, BeMe:18}.
\item The $hp$-boundary-element method ($hp$-BEM), where, under a polynomial-boundedness assumption on the solution operator,
 the results of \cite{LoMe:11, Me:12} establish that the boundary-element method when $\hFEM\sim k^{-1}$ and $p \sim \log k$ does not suffer from the pollution effect.
\item The multiscale finite-element method of \cite{GaPe:15}, \cite{BrGaPe:15}, 
\cite{Pe:17}, which, under the assumption that the resolvent is polynomially bounded in $k$,
computes solutions that are uniformly accurate in $k$ but with a total number of degrees of freedom $\sim k^n$, provided that a certain oversampling parameter grows logarithmically with $k$. 
\een
The next two subsections give the details of the results outlined in Points 1 and 2 above for obstacles with strong trapping (for brevity we do not give the details of the results in Point 3).

\subsubsection{Quasioptimality of $hp$-FEM for trapping domains for most frequencies}\label{sec:FEMqo}

Given $R> \max_{\bx\in \partial \obstacle_+}|\bx|$, let $\obstacle_R:= \obstacle_+\cap B_R$, and let the Hilbert space $V_R:= \{w|_{\obstacle_R}: w\in H^1_{\text{loc}}(\Omega_+)\mbox{ and }\gamma w=0\}$. A standard reformulation of the EDP, and the starting point for discretisation by FEMs, is the variational problem 
\beq\label{eq:vp}
\text{ find } u_R \in V_R \tst \,\,a(u_R,v) = F(v)\,\, \tfa v\in V_R,
\eeq
where 
\beqs
a(u,v) := \int_{\obstacle_R} (\nabla u\cdot \overline{\nabla v} - k^2 u \bar v)\,\rd x - \int_{\partial B_R} \overline{\gamma  v} \,T_R \gamma u \, \rd s,
\quad\tand \quad F(v):= \int_{\obstacle_R} f \,\overline{v}\rd \bx,
\eeqs
where $T_R$ is the DtN map for the exterior problem with obstacle $B_R$; see, e.g., \cite[Eq.~3.5 and 3.6]{ChMo:08}, 
\cite[Eq.~3.7 and 3.10]{MeSa:10} for the definition of $T_R$ in terms of Hankel functions and polar coordinates (when $d=2$)/spherical polar coordinates (when $d=3$).
This set-up implies that the solution $u_R$ to the variational problem \eqref{eq:vp} is $u|_{\obstacle_R}$, where $u$ is the solution of the EDP described in \S\ref{sec:1.1}. Let $C_{\rm cont}$ be the continuity constant of the sesquilinear form $a(\cdot,\cdot)$ in the norm $\|\cdot\|_{H^1_k(\obstacle_R)}$, i.e.~$a(u,v)\leq C_{\rm cont} \|u\|_{H^1_k(\obstacle_R)} \|v\|_{H^1_k(\obstacle_R)}$ for all $u,v \in V_R$ and for all $k\geq k_0$; by the Cauchy-Schwarz inequality and the bound on $T_R$ in \cite[Lemma 3.3]{MeSa:10}, $C_{\rm cont}$ is independent of $k$ (but dependent on $k_0$).

Let $\mathcal{T}_\hFEM$ be a quasi-uniform triangulation of $\obstacle_R$ in the sense of \cite[Assumption 5.1]{MeSa:11}, with $\hFEM:=\max_{K\in \mathcal{T}_{\hFEM}}\mathrm{diam}(K)$ the maximum element diameter. Let $\mathcal{S}_0^{p,1}(\mathcal{T}_{\hFEM}):= \mathcal{S}^{p,1}(\mathcal{T}_{\hFEM})\cap V_R$, where $\mathcal{S}^{p,1}(\mathcal{T}_{\hFEM})$ is the space of continuous, piecewise polynomials of degree $\leq p$ on the triangulation $\mathcal{T}_{\hFEM}$  \cite[Eq.~5.1]{MeSa:11}. The $hp$-FEM then seeks $u_{hp}$ -- 
an approximation of $u_R$ in the subspace $\mathcal{S}_0^{p,1}(\mathcal{T}_{\hFEM})$ -- as the solution  of
\beq\label{eq:Galerkin}
\text{find } u_{hp} \tst \,\,a(u_{hp},v_{hp}) = F(v_{hp}) \,\, \mbox{for all }v_{hp}\in \mathcal{S}_0^{p,1}(\mathcal{T}_{\hFEM}).
\eeq
Theorem \ref{thm:intro1} implies that the polynomial-boundness assumption (\cite[Assumption 4.18]{MeSa:11}) in the analysis of the $hp$-FEM in \cite{MeSa:11} is satisfied for most frequencies, and \cite[Theorem 5.18]{MeSa:11} then implies the following.

\begin{corollary}[$k$-independent quasioptimality of $hp$-FEM for most frequencies]
Let $\obstacle_+$ be as in Theorem \ref{thm:intro1}, and assume further that
 $\partial\obstacle_+$ is analytic. 
 Given $k_0>0$, $\delta>0$, and $\eps>0$, there exists $C_j=C_j(k_0, \delta, \eps, n, \obstacle_-)>0$, $j=1,2$, 
and a set $J\subset [k_0,\infty)$ with $|J|\leq \delta$ such that, if 
\beq\label{eq:dof}
\frac{kh}{p} \leq C_1 \quad\tand\quad p \geq 1 + C_2 \log (2+k),
\eeq
then, for all $k\in [k_0,\infty)\setminus J$, the Galerkin solution $u_{hp}$ defined by \eqref{eq:Galerkin} exists, is unique, and satisfies the quasi-optimal error estimate
\beq\label{eq:qo}
\N{u_R- u_{hp}}_{H^1_k(\obstacle_R)} \leq 2\big( 1  + C_{\rm cont}\big) \min_{v_{hp} \in \mathcal{S}_0^{p,1}(\mathcal{T}_{\hFEM})}\N{u_R - v_{hp}}_{H^1_k(\obstacle_R)}.
\eeq
\end{corollary}

In this corollary we assumed that $\partial\obstacle_+$ is analytic; this is so we could directly apply \cite[Theorem 4.18]{MeSa:11}, but we highlight that analogous quasi-optimality results under polynomial-boundedness of the resolvent are obtained for non-convex polygonal domains in \cite{EsMe:12}.

The significance of the quasioptimality results for the $hp$-FEM in \cite{MeSa:10, MeSa:11, EsMe:12} is that they show that the $hp$-FEM does not suffer from the pollution effect, in that the constant $2(1+ C_{\rm cont})$ on the right-hand side of \eqref{eq:qo} is independent of $k$, and $h$ and $p$ satisfying \eqref{eq:dof} can be chosen so that the total number of degrees of freedom (i.e.~the dimension of the subspace $\mathcal{S}_0^{p,1}(\mathcal{T}_{\hFEM})$) grows like $k^n$ (see \cite[Remark 5.9]{MeSa:11} for more details). The resolvent estimate of Theorem \ref{thm:intro1} now shows that this property is enjoyed even for strongly trapping obstacles, at least for most frequencies.

\subsubsection{Quasioptimality of $hp$-BEM for trapping domains for most frequencies}

\paragraph{Integral equations for the exterior Dirichlet problem} In this subsection, we let $u\in H^1_{\mathrm{loc}}(\obstacle_+)$ be a solution to the Helmholtz equation $\Delta u + k^2 u = 0$ in $\obstacle_+$ that satisfies the Sommerfeld radiation condition \eqref{eq:src} and the boundary condition $\gamma u=g$ for $g\in H^1(\partial\obstacle_+)$ (note that if the data $g$ arises from plane-wave or point-source scattering, this regularity of $g$ is guaranteed; see \cite[Definition 2.11]{ChGrLaSp:12}).

We now briefly state the standard second-kind integral-equation formulations of this problem. Let $\Phi_k(\bx,\by)$ be the fundamental solution of the Helmholtz equation given by 
\beqs
\Phi_k(\bx,\by)=\displaystyle\frac{\ri}{4}H_0^{(1)}\big(k|\bx-\by|\big), \,\,d=2,\quad\quad \Phi_k(\bx,\by) = \frac{\re^{\ri k |\bx-\by|}}{4\pi |\bx-\by|}, \,\,d=3
\eeqs
and let $S_k$, $D_k$, and $D'_k$ be the single-layer,  double-layer and adjoint-double-layer operators defined by 
\begin{align*}
S_k \phi(\bx) := &\int_\bound \Phi_k(\bx,\by) \phi(\by)\,\rd s(\by), \quad
D'_k \phi(\bx) := \int_\bound \frac{\partial \Phi_k(\bx,\by)}{\partial n(\bx)}  \phi(\by)\,\rd s(\by),
\end{align*}
\beqs
D_k \phi(\bx) := \int_\bound \frac{\partial \Phi_k(\bx,\by)}{\partial n(\by)}  \phi(\by)\,\rd s(\by) \quad\tfor \phi\in\LtG \,\tand \, \bx\in\bound.
\eeqs
The standard second-kind combined-field ``direct" formulation (arising from Green's
integral representation) and ``indirect" formulation (arising from an
ansatz of layer potentials not related to Green's integral
representation) are, respectively, 
\beq\label{eq:direct}
\opA \partial_{\nu} u = f_{k,\eta}\quad\tand\quad \opABW \phi = g,
\eeq
where 
\beq\label{eq:scpo}
\opA := \half I + D'_k -\ri\eta S_k, \qquad \opABW := \half I + D_k -\ri\eta S_k,
\eeq
where $\eta \in \Rea\setminus \{0\}$ is an arbitrary coupling parameter. 
In \eqref{eq:direct} the unknown $f_{k,\eta}$ is given in terms of the Dirichlet data $g$ by, e.g., \cite[Eq.~2.69 and 2.114]{ChGrLaSp:12}, and in the indirect formulation the solution $u$ can be recovered from the potential $\phi$; see, e.g., \cite[Eq.~2.70]{ChGrLaSp:12}. 

The operators $\opABWinv$ and $\opAinv$ can be expressed in terms of (i) the exterior Dirichlet-to-Neumann map, and (ii) the interior impedance-to-Dirichlet map, see \cite[Theorem 2.33]{ChGrLaSp:12}, and therefore bounds on $\opABWinv$ and $\opAinv$ can be obtained from bounds on these maps \cite{ChMo:08}, \cite{Sp:14}, \cite{BaSpWu:16}, \cite{ChSpGiSm:17}. Inputting into \cite[Lemma 6.3]{ChSpGiSm:17} the bound on (i) from Corollary \ref{cor:DtN} and the bounds on (ii) from \cite[Corollary 1.9]{BaSpWu:16}, \cite[Corollary 4.7]{Sp:14}, we obtain the following corollary. For simplicity, we only state bounds on the $L^2(\bound)\rightarrow L^2(\bound)$ norms of $\opABWinv$ and $\opAinv$, but we highlight that bounds in the spaces $H^s(\bound)$ and $H^s_k(\bound)$ can also be obtained; see \cite[Lemma 6.3]{ChSpGiSm:17}.

\begin{corollary}[Bounds on $\opAinv$ and $\opABWinv$ for most frequencies]\label{cor:opA}
Let $\obstacle_+$ be as in Theorem \ref{thm:intro1}.

(i) Given $k_0>0$, $\delta>0$, and $\eps>0$, there exists $C''=C''(k_0, \delta, \eps, n, \obstacle_-)>0$ and a set $J\subset [k_0,\infty)$ with $|J|\leq \delta$ such, if $\eta = ck$, for some $c\in \R\setminus\{0\}$, then 
\beq\label{eq:Ainv}
\normAinv_{\LtGt}= \normABWinv_{\LtGt} \leq C'' k^{5n/2 +1/2 +\eps} 
\eeq
for all $k \in [k_0,\infty)\setminus J.$

(ii) If the boundaries of the (finite number of) disjoint components of $\obstacle_-$ are each piecewise smooth, then the exponent in \eqref{eq:Ainv} reduces to $5n/2 + 1/4+\eps$.

(iii) If \emph{either} the components are star-shaped with respect to a ball \emph{or} the boundaries are $C^\infty$ then the exponent reduces to $5n/2+\eps$.
\end{corollary}

\paragraph{The $hp$-BEM}
For simplicity of exposition, we now focus on the Galerkin method applied to the direct equation $\opA \partial_{\nu} u = f_{k,\eta}$, but everything below holds also for the indirect equation $\opABW \phi =g$.
Assume that $\bound$ is analytic, and that $\mathcal{T}_{\hFEM}$ is a quasi-uniform triangulation with mesh size $h$ of $\Gamma$ in the sense of \cite[Definition 3.15]{LoMe:11}. Let $\mathcal{S}^p(\mathcal{T}_{\hFEM})$ denote the space of continuous, piecewise polynomials of degree $\leq p$ on the triangulation $\mathcal{T}_{\hFEM}$. The $hp$-BEM then seeks $(\partial_{\nu}u)_{hp}$ -- an approximation of $\partial_{\nu} u$ in the subspace $\mathcal{S}^p(\mathcal{T}_{\hFEM})$ -- as the solution of 
\beq\label{eq:GalerkinBEM}
\big(\opA (\partial_{\nu} u)_{hp},v_{hp}\big)_\Gamma = (f_{k,\eta},v_{hp})_\Gamma \quad \mbox{for all } v_{hp}\in \mathcal{S}^p(\mathcal{T}_{\hFEM}),
\eeq
where $(\cdot,\cdot)_\Gamma$ denotes the inner product on $L^2(\Gamma)$.

Corollary \ref{cor:opA} implies that the polynomial-boundness assumption (\cite[Eq.~3.24]{LoMe:11}) in the analysis of the $hp$-BEM in \cite{LoMe:11} is satisfied for most frequencies, and \cite[Corollary 3.18]{LoMe:11} then implies the following.

\begin{corollary}[$k$-independent quasi-optimality of the $hp$-BEM for most frequencies]
Let $\obstacle_+$ be as in Theorem \ref{thm:intro1}, and assume further that
 $\partial\obstacle_+$ is analytic. 
Assume that $\eta = ck$, for some $c \in \Rea\setminus\{0\}$.
Given $k_0>0$, $\delta>0$, and $\eps>0$, there exists $C_j=C_j(k_0, \delta, \eps, n, \obstacle_-, c)>0$, $j=1,2$, $C_3= C_3(\obstacle_-)>0$, and a set $J\subset [k_0,\infty)$ with $|J|\leq \delta$ such that, if $k\geq k_0$ and \eqref{eq:dof} holds, 
then, for all $k\in [k_0,\infty)\setminus J$, the Galerkin solution $(\partial_{\nu}u)_{hp}$ defined by \eqref{eq:GalerkinBEM} exists, is unique, and satisfies the quasi-optimal error estimate
\beqs
\N{(\partial_{\nu} u)_{hp} - v_{hp}}_{L^2(\Gamma)} \leq C_3 \inf_{v_{hp}\in \mathcal{S}^p(\mathcal{T}_{\hFEM})} \N{\partial_{\nu} u - v_{hp}}_{L^2(\Gamma)}.
\eeqs
\end{corollary}

The significance of the quasioptimality results for the $hp$-BEM in \cite{LoMe:11} is that they show that the $hp$-BEM does not suffer from the pollution effect, in that the constant $C_3$ in \eqref{eq:qo} is independent of $k$, and $h$ and $p$ satisfying \eqref{eq:dof} can be chosen so that the total number of degrees of freedom  grows like $k^{n-1}$ (see \cite[Remark 3.19]{LoMe:11} for more details). Just as in the $hp$-FEM case, the resolvent estimate of Theorem \ref{thm:intro1} (via Corollary \ref{cor:opA}) now shows that this property is enjoyed even for strongly trapping obstacles, at least for most frequencies.

\section{Recap of the black-box scattering framework}\label{sec:blackbox}

\subsection{Abstract framework}

We now briefly recap the abstract framework of \emph{black-box scattering} introduced in \cite{SjZw:91};
for more details, see  the comprehensive presentation in \cite[Chapter 4]{DyZw:19}. \footnote{In this section, we recap the black-box framework for non-semiclassically-scaled operators, as in \cite[\S2]{TaZw:00}. We highlight that \cite[Chapter 4]{DyZw:19} deals with semiclassically-scaled operators, but transferring the results from \cite[Chapter 4]{DyZw:19} into the former setting is straightforward.}

Let $\mathcal{H}$ be an Hilbert space with an orthogonal decomposition
\[
\mathcal{H}=\mathcal{H}_{R_{0}}\oplus L^{2}(\mathbb{R}^{n}\backslash B(0,R_{0})),
\]
and let $P$ be a self adjoint operator $\mathcal{H}\rightarrow\mathcal{H}$
with domain $\mathcal{D}\subset\mathcal{H}$ (so, in particular, $\mathcal{D}$ is dense in $\mathcal{H}$). We require that the operator $P$ be $-\Delta$ outside 
$\mathcal{H}_{R_{0}}$ in the sense that 
\beq\label{eq:bbreq1}
 \boldsymbol{1}_{\mathbb{R}^{n}\backslash B(0,R_{0})}P=-\Delta\vert_{\mathbb{R}^{n}\backslash B(0,R_{0})}, \quad
\boldsymbol{1}_{\mathbb{R}^{n}\backslash B(0,R_{0})}\mathcal{D} \subset H^{2}(\mathbb{R}^{n}\backslash B(0,R_{0})),
\eeq
We further assume that
\[
v\in H^{2}(\mathbb{R}^{n}),\ v\vert_{B(0,R_{0}+\eps)}=0\quad \text{implies  that}\quad v\in\mathcal{D},
\]
and that
\beq\label{eq:bbreq2}
\boldsymbol{1}_{ B(0,R_{0})}(P+\ri)^{-1}\text{ is compact from } \cH \rightarrow \cH.
\eeq
Under these assumptions, the resolvent 
\beq\label{eq:resolvent1}
R(k)=(P-k^{2})^{-1}:\mathcal{H}\rightarrow\mathcal{D}
\eeq
is meromorphic for $\Imag \,k>0$ and extends to a meromorphic
family of operators of $\mathcal{H_{\rm comp}}\rightarrow\mathcal{D}_{\rm loc}$
in the whole complex plane when $n$ is even and in the logarithmic
plane when $n$ is odd \cite[Theorem 4.4]{DyZw:19}. The poles of $(P-k^{2})^{-1}$ are called
the \emph{resonances} of $P$, and we denote them by $\Res\,P$. 

To study the resonances of $P$, we define a \emph{reference
operator} $P^{\text{\ensuremath{\sharp}}}$ associated to $P$ but
acting in a compact manifold: we glue our black box
into a torus in place of $\mathbb{R}^n$. For a precise definition, see \cite[\S4.3]{DyZw:19}, but we note here that $P^{\text{\ensuremath{\sharp}}}$
is defined in
\[
\mathcal{H}^{\sharp}=\mathcal{H}_{R_{0}}\oplus L^{2}((\mathbb{R}/R_1\mathbb{Z})^{n}\backslash B(0,R_{0})),\ R_1>R_{0},
\]
and can be thought of as $P$ in $\mathcal{H}_{R_{0}}$ and $-\Delta$ in
$(\mathbb{R}/R_1\mathbb{Z})^{n}\backslash B(0,R_{0})$. We assume that
the eigenvalues of $P^{\text{\ensuremath{\sharp}}}$ satisfy the \emph{polynomial
growth of eigenvalues condition }
\begin{equation}\label{eq:gro}
N\big(P^{\#},[-C,\lambda]\big)=O(\lambda^{n^{\#}/2}),
\end{equation}
where $n^{\#}\geq n$ and $N(P^{\#},I)$ is the number of eigenvalues of
$P^{\sharp}$ in the interval $I$, counted with their multiplicity \footnote{Note that here we use the convention in \cite{TaZw:00} of counting eigenvalues in $[-C, \lambda]$, instead of using the convention of \cite[Equation 4.3.10]{DyZw:19} of counting eigenvalues in $[-C,\lambda^2]$.}
When $n^{\#}=n$, the asymptotics \eqref{eq:gro} correspond a Weyl-type upper bound, 
and thus \eqref{eq:gro} can be thought of as a weak Weyl law. 
One can then show that 
the
resonances of $P$ grow in the same way, that is
\begin{equation}
N(P,r,\theta)\lesssim r^{n^{\#}}
\label{eq:count}
\end{equation}
where $N(P,r,\theta)$ is the number of resonances of $P$ (counted
with their multiplicity) in the sector $\left\{ |z|\leq r,\ \arg z<\theta\right\} $, and the omitted constant in \eqref{eq:count} depends on $\theta$; see \cite{SjZw:91}, \cite{Vo:92}, \cite{Vo:94}, \cite[Theorem 4.13]{DyZw:19} for this result for resonances in the disc of radius $r$ and \cite[Text after Eq.~2.10]{Sj:97}, \cite[Eq.~2.1]{TaZw:00} for resonances in a sector.

In the proof of Theorem \ref{thm:intro1} (and its black-box analogue Theorem \ref{thm:bb1} below) it is convenient to work with the semiclassical
operator $h^{2}P$, where $h>0$ is a small parameter. We define the \emph{semiclassical resolvent}, $R(z,h)$, by
\beq\label{eq:resolvent}
R(z,h):=(h^{2}P-z)^{-1},
\eeq
and we let $\mathcal{R}$
be the set of the poles of the meromorphic
continuation of $R(z,h)$, i.e., the semiclassical resonances.
Observe that
\beqs
z\in\mathcal{R}(h)\text{  implies  } h^{-1}z^{1/2}\in\Res\,P\text{,  \quad and  }k\in\Res\,P\text{  implies  } h^{2}k^{2}\in\mathcal{R}(h). 
\eeqs

\subsection{Scattering problems fitting in the black-box framework}\label{sec:bbexamples}

Scattering problems fitting in the black-box framework include scattering by impenetrable and penetrable obstacles, scattering by a compactly supported potential (i.e.~$P=-\Delta+V$), 
scattering by elliptic compactly-supported perturbations of the Laplacian, and scattering on finite volume surfaces; see \cite[\S4.1]{DyZw:19}.

Here we focus on scattering by impenetrable and penetrable obstacles. In the literature, these are usually placed in the black-box framework when the boundary of the obstacle is $C^\infty$; here we show that obstacles with Lipschitz boundaries can also be put into this framework.

\begin{lemma}[Scattering by an impenetrable Dirichlet or Neumann Lipschitz obstacle in black-box framework]\label{lem:obstacle}
Let $\obstacle_-\subset\Rea^n$, $n\geq 2$ be a bounded open set with Lipschitz boundary such that the open complement $\obstacle_+:=\Rea^n\setminus\overline{\obstacle_-}$ is connected and such that $\obstacle_- \subset B(0,R_0)$. Let 
$A \in C^{0,1} (\obstacle_+ , \Rea^{d\times d})$ be such that $\supp(I -A)\subset B(0,R_0)$, $A$ is symmetric, and there exists $A_{\min}>0$ such that
\beq\label{eq:Aelliptic}
\big(A(\bx) \bxi\big) \cdot\overline{ \bxi} 
\geq  A_{\min} |\bxi|^2
\quad\text{ for almost every }\bx \in \obstacle_+ \text{ and for all } \bxi\in \Com^d.
\eeq
Let $\nu$ be the unit normal vector field on $\partial \obstacle_-$  pointing from $\obstacle_-$ into $\obstacle_+$, and let 
$\partial_{\nu,A}$ denote the corresponding conormal derivative
defined by, e.g., \cite[Lemma 4.3]{Mc:00} (recall that this is such that, when $v \in H^2(\obstacle_{+})$, $\partial_{\nu, A}v = \nu \cdot \gamma (A\nabla v)$).
Then the operator 
$Pv := -\nabla \cdot \big(A\nabla v)$
with either one of the domains 
\beqs
\cD_D :=\Big\{ v\in H^1(\obstacle_+), \, \nabla\cdot \big(A\nabla v\big) \in L^2(\obstacle_+), \, \gamma v=0\Big\}
\eeqs
or 
\beqs
\cD_N :=\Big\{ v\in H^1(\obstacle_+), \, \nabla\cdot \big(A\nabla v\big) \in L^2(\obstacle_+), \, \partial_{\nu,A}v=0\Big\}
\eeqs
fits into the black-box framework with 
\beqs
\cH = L^2(\obstacle_+), \quad\tand\quad \cH_{R_0} = L^2\big( B(0,R_0) \cap \obstacle_+\big).
\eeqs
Furthermore the corresponding reference operator $P^{\#}$ (defined precisely in \cite[\S4.3]{DyZw:19}) satisfies \eqref{eq:gro} with $n^{\#}=n$.
\end{lemma}

\bpf
Since $C^\infty$ functions with compact support are both dense in
$L^2(\obstacle_+)$ and contained in $\cD_D$ and $\cD_N$ when $A$ is
Lipschitz, $\cD_D$ and $\cD_N$ are both dense in $\cH$.  
The definitions of $\cD_{D/N}$ imply that $P$ is self-adjoint; the definitions of $\cD_{D/N}$ and $\cH$ imply that  $P: \cD_{D/N} \rightarrow \cH$ and that the resolvent $R: \cH\rightarrow \cD_{D/N}$. The operator $P$ is then self-adjoint by Green's second identity (valid in Lipschitz domains by, e.g., \cite[Theorem 4.4(iii)]{Mc:00}). The first condition in \eqref{eq:bbreq1} is satisfied since $\supp (I-A) \subset B(0,R_0)$, and the second condition in \eqref{eq:bbreq1} is satisfied due to interior regularity of the Laplacian (see, e.g., \cite[Theorem 4.16]{Mc:00}). The condition \eqref{eq:bbreq2} follows from the compact embedding of $H^1(B(0,R_0)\cap \obstacle_+)$ in $L^2(B(0,R_0)\cap \obstacle_+)$; see, e.g., \cite[Theorem 3.27]{Mc:00}. 
The polynomial
growth of eigenvalues condition \eqref{eq:gro} follows from results about heat-kernel asymptotics from \cite{No:97}; see Lemma \ref{lem:Weyl1}.

Note that in \cite[Chapter 4]{DyZw:19} (our default reference for the black-box framework), the (semiclassically-scaled) norm defined by $\|u\|_{\cD}^2:= \|u\|_{\cH}^2 + h^{4}\|Pu \|_{\cH}^2$ is placed on $\cD$; in our setting this would correspond to the norm squared being $\|u\|_{L^2}^2 + h^{4}\|\nabla\cdot(A\nabla u) \|_{L^2}^2$. However, the results in \cite{DyZw:19} also hold with the norm squared being $\|u\|_{L^2}^2 + h^2\|\nabla u\|^2_{L^2}+h^4\|\nabla\cdot(A\nabla u) \|_{L^2}^2$. Indeed, the only place the form of the norm on $\cD$ is used in \cite{DyZw:19} is in the bounds of \cite[Lemma 4.3]{DyZw:19}, which are used in the proof of meromorphic continuation of the resolvent \cite[Theorem 4.4]{DyZw:19}. However, the bounds in \cite[Lemma 4.3]{DyZw:19}
hold also (at least in this obstacle setting) with the norm squared
being $\|u\|_{L^2}^2 + h^2\|\nabla u\|^2_{L^2}+h^4\|\nabla\cdot(A\nabla u)
\|_{L^2}^2$, since control of the $\nabla u$ term follows from control
of $u$ and $\Delta u$ via, e.g., Green's identity. 
\epf 

\bre[Exterior Dirichlet or Neumann scattering problem]
With $P$ and $A$ as in Lemma \ref{lem:obstacle}, given $f\in L^2(\obstacle_+)$ with compact support and $k>0$, $u:= R(k)f$ satisfies either one of the boundary-value problems: $u \in H^1_{\rm loc}(\obstacle_+)$,
\beqs
\nabla \cdot (A \nabla u) +k^2 u =-f \quad\tin \obstacle_+, \quad \text{ either } \gamma u =0 \text{ or } \partial_\nu u=0 \ton \partial \obstacle_+,
\eeqs
and the radiation condition \eqref{eq:src} at infinity.
\ere

\begin{lemma}[Scattering by an penetrable Lipschitz obstacle in black-box framework]\label{lem:transmission}
Let $\obstacle_-$ be as in Lemma \ref{lem:obstacle}.
Let 
$A \in C^{0,1} (\Rea^d, \Rea^{d\times d})$ be such that $\supp(I -A)\subset B(0,R_0)$, $A$ is symmetric, and there exists $A_{\min}>0$ such that
\eqref{eq:Aelliptic} holds (with $\obstacle_+$ replaced by $\Rea^d$).
Let $\nu$ be the unit normal vector field on $\partial \obstacle_-$  pointing from $\obstacle_-$ into $\obstacle_+$, and let 
$\partial_{\nu,A}$ the corresponding conormal derivative from either $\obstacle_-$ or $\obstacle_+$.
For $D$ an open set, let $H^1(D,\nabla\cdot(A\nabla\cdot)):= \{ v : v\in H^1(D), \nabla\cdot(A\nabla v)\in L^2(D)\}$.
Let $c,\alpha>0$ and set 
\beq\label{eq:measure}
\mathcal{H}_{R_{0}}=L^{2}\big(\mathcal{O},c^{-2}\alpha^{-1}\rd x\big)\oplus L^{2}\big(B(0,R_{0})\backslash\overline{\mathcal{O}}\big),
\eeq
so that 
\beqs
\mathcal{H}= L^{2}\big(\mathcal{O};c^{-2}\alpha^{-1}\rd x\big)\oplus L^{2}\big(B(0,R_{0})\backslash\overline{\mathcal{O}}\big) \oplus L^{2}\big(\mathbb{R}^{n}\backslash B(0,R_{0})\big).
\eeqs
Let,
\begin{align}\nonumber
\cD :=&\Big\{ v= (v_1,v_2,v_3) \text{ where } v_1 \in H^1\big(\obstacle_-, \nabla\cdot(A\nabla\cdot)\big),\\ \nonumber
&\quad v_2 \in H^1\big(B(0,R_0)\setminus \overline{\obstacle_-}, \nabla\cdot(A\nabla\cdot)\big), \quad
 v_3 \in H^1\big(\Rea^n\setminus \overline{B(0,R_0)}, \Delta\big), \\ \nonumber
 & \quad \gamma v_1 = \gamma v_2\quad\tand\quad \partial_{\nu,A} v_1 = \alpha\, \partial_{\nu,A} v_2 \text{ on } \partial \obstacle_-, \tand\\ 
 & \quad \gamma v_2 = \gamma v_3\quad\tand \quad \partial_{\nu,A} v_2 = \partial_{\nu,A} v_3 \text{ on } \partial B(0,R_0) \,\,\Big\}
 \label{eq:domain_transmission}
\end{align}
(observe that the conditions on $v_2$ and $v_3$ on $\partial B(0,R_0)$ in the definition of $\cD$ are such that $(v_2,v_3) \in H^1(\Rea^n\setminus \overline{\obstacle_-}, \nabla\cdot(A\nabla\cdot))$). 
Then the operator 
\beqs
Pv:=-\big(c^{2}\nabla\cdot(A\nabla  v_{1}),\nabla\cdot(A\nabla v_{2}),\Delta v_{3}\big),
\eeqs
defined for $v=(v_1,v_2,v_3)$, fits in the the black-box framework, and the 
the corresponding reference operator $P^{\#}$ (defined precisely in \cite[\S4.3]{DyZw:19}) satisfies \eqref{eq:gro} with $n^{\#}=n$.
\end{lemma}

\bpf
The domain $\cD$ contains $C^\infty$ functions that are zero in a neighbourhood of $\partial \obstacle_-$, and these are dense in $L^2(\Rea^n)$.
The scalings in the measure imposed on $\obstacle_-$ in \eqref{eq:measure} imply that $P$ is self-adjoint by Green's identity. 
The conditions \eqref{eq:bbreq1} and \eqref{eq:bbreq2} are satisfied by the same arguments in Lemma \ref{lem:obstacle}.
The proof that the corresponding reference operator $P^{\#}$ satisfies \eqref{eq:gro} with $n^{\#}=n$ is given in Lemma \ref{lem:Weyl2}.
The remarks in the proof of Lemma \ref{lem:obstacle} about the norm applied on $\cD$ in \cite[Chapter 4]{DyZw:19} also apply here.
\epf

\bre[Scattering by a penetrable obstacle (a.k.a.~the transmission problem)]\label{rem:transmission}
With $\obstacle_-$, $A$, and $P$ as in Lemma \ref{lem:transmission}, given $f\in L^2(\Rea^n)$ with compact support and $c, \alpha, k>0$, and let $u:= R(k)f$. Then,
with the notation $\uin = u|_{\obstacle_-}$ ($=u_1$ in the notation of Lemma \ref{lem:obstacle}) and $\uout=u|_{\obstacle_+}$ ($= (u_2,u_3)$), $u$ satisfies the 
boundary-value problem: $u\in H^1_{\rm loc}(\Rea^n \setminus \partial\obstacle_-)$,
\begin{align}\label{eq:transmission1}
&\nabla\cdot(A\nabla \uin) + \frac{k^2}{c^2}\uin = -f \,\, \tin \obstacle_-,\quad
\nabla\cdot(A\nabla \uout) + k^2  \uout =- f \,\, \tin \obstacle_+,\\ \label{eq:transmission2}
&\qquad\gamma \uin = \gamma \uout \quad\tand \quad\partial_{\nu,A} \uin = \alpha \,\partial_{\nu,A} \uout \quad\ton \partial \obstacle_-,\\
&\qquad \uout \text{ satisfies the radiation condition \eqref{eq:src}}.\label{eq:transmission3}
\end{align}

By rescaling, any other transmission problem with constant real coefficients can be written as \eqref{eq:transmission1}-\eqref{eq:transmission3}; see \cite[Definition 2.3 and paragraph immediately afterwards]{MoSp:19}.
\ere

\bre[Trapping by penetrable obstacles]\label{rem:trappingtransmission} When $c<1$ and 
$\partial \obstacle_-$ is $C^\infty$ with strictly positive curvature, then the boundary-value problem \eqref{eq:transmission1}-\eqref{eq:transmission3} is trapping; see \cite{PoVo:99}, \cite{St:06}, \cite{Ca:12}, \cite{CaLePa:12}, \cite[\S6]{MoSp:19}.
\ere

\section{Polynomial resolvent estimates away from ``bad" frequencies (including the proof of Theorem \ref{thm:intro1})}
\label{sec:proofs}

For completeness, we state the two main ingredients of our proofs, namely (i) 
the semiclassical maximum principle of \cite[Lemma 2]{TaZw:98}, \cite[Lemma 4.2]{TaZw:00} (see also \cite[Lemma 7.7]{DyZw:19}), and (ii) exponential resolvent bounds away from resonances from \cite[Lemma 1]{TaZw:98}, \cite[Proposition 4.3]{TaZw:00} (see also \cite[Theorem 7.5]{DyZw:19}).

\begin{theorem}[Semiclassical maximum principle \cite{TaZw:98, TaZw:00}]
\label{thm:scmp}
Let $\cH$ be an Hilbert space and $z\mapsto Q(z,h)\in\mathcal{L}(\cH)$
an holomorphic family of operators in a neighbourhood of
\beqs
\Omega(h):=\big(w-2a(h),w+2a(h)\big)+i\big(-\delta(h)h^{-L},\delta(h)\big),
\eeqs
where 
\beq\label{eq:restrict1}
0<\delta(h)<1,\qquad \tand \quad a(h)^{2}\geq Ch^{-3L}\delta(h)^{2}
\eeq
for some $L>0$ and $C>0$. Suppose that
\begin{align}\label{eq:expbound}
\Vert Q(z,h)\Vert_{\cH\rightarrow\cH}&\leq\exp(Ch^{-L}),\qquad z\in\Omega,\\ \label{eq:absbound}
\Vert Q(z,h)\Vert_{\cH\rightarrow\cH}&\leq\frac{1}{\Imag\,z},\qquad \Imag\,z>0,\quad z\in\Omega.
\end{align}
Then
\beq\label{eq:scmp1}
\Vert Q(z,h)\Vert_{\cH\rightarrow\cH}\leq \delta(h)^{-1}\exp(C+1),\quad \tfa z\in\big[w-a(h),w+a(h)\big].
\eeq
\end{theorem}

\begin{proof}[References for proof]
Let $f,g\in \cH$ with $\|f\|_{\cH}=\|g\|_{\cH}=1$, and let
\[
F(z,h):=\big\langle Q(z+w,h)g,f\big\rangle_{\cH}.
\]
The result \eqref{eq:scmp1} follows from the ``three-line theorem in a rectangle'' (a consequence of the maximum principle) stated as \cite[Lemma D.1]{DyZw:19} applied to the holomorphic family $(F(\cdot,h))_{0<h\ll1}$
with
\begin{align*}
R=2a(h),\qquad \delta_{+}=\delta(h),\qquad \delta_{-}=\delta(h)h^{-L},\\
M=M_{-}=\exp(Ch^{-L}),\qquad M_{+}=\delta(h)^{-1}.
\end{align*}
\end{proof}

\bth[Bounds on the resolvent away from resonances \cite{TaZw:98, TaZw:00}]\label{thm:TZkey}
Let $P$ satisfy the assumptions in \S\ref{sec:blackbox}  and let $n^{\#}$ be the exponent in the condition \eqref{eq:gro}.
Let $\Omega\Subset\left\{ \Real\,z>0\right\} $ be a precompact
neighbourhood of  $E\in \mathbb{R}^+$, 
Let $h \mapsto g(h)$ be a positive function. Then 
there exist $h_0>0$ and $C_1>0$ (both depending on $\Omega$) such that, for $0<h<h_0$, the resolvent \eqref{eq:resolvent} satisfies 
\beq\label{eq:TZkey1}
\N{\chi R(z,h)\chi }_{\cH\rightarrow \cH}
\leq C_1\exp\left(C_1 h^{- {n^{\#}}} \log \left(\frac{1}{g(h)}\right) \right)\,\,\tfa\,\,
z \in \Omega \setminus \bigcup_{z_j \in \cR_P} D(z_j, g(h))
\eeq
(where $D(z_j, g(h))$ is the open disc of radius $g(h)$ centred at $z_j\in \Com$).
\end{theorem}

The significance of Theorem \ref{thm:TZkey} is that it provides one of the two bounds needed to apply the semiclassical maximum principle to the resolvent $R(z,h)$, namely \eqref{eq:expbound}. The second bound, \eqref{eq:absbound}, is given by the following bound on the resolvent, valid when
$P$ satisfies the assumptions in \S\ref{sec:blackbox}, and $\Imag z>0$, 
\beq\label{eq:absbound2}
\big\| R(z,h)\big\|_{\cH\rightarrow\cH} \leq \frac{1}{\Imag z}
\eeq
(a simple way to prove this is by taking the inner product of the equation $(h^2P -z)u=f$ with $u$ and using the self-adjointness of $P$)

\begin{theorem}[Black-box analogue of Theorem \ref{thm:intro1}]\label{thm:bb1}
Let $P$ satisfy the assumptions in \S\ref{sec:blackbox} and let $n^{\#}$ be the exponent in the condition \eqref{eq:gro}.
Then, given $k_0>0$, $\delta>0$, and $\eps>0$, there exists a $C=C(k_0,\delta,\eps, n^{\#})>0$ and a set $J$ with $|J|\leq \delta$ such that the resolvent \eqref{eq:resolvent1} satisfies
\beq\label{eq:thmbb11}
\Vert\chi R(k)\chi\Vert_{\cH\rightarrow \cH}\leq Ck^{5n^{\#}/2+ \eps},
\quad \tfa k \in [k_0,\infty)\backslash J.
\eeq
\end{theorem}

\begin{proof}
Let $\Omega\Subset\left\{ \Real\,z>0\right\} $ be a 
neighbourhood of  $E\in \mathbb{R}^+$, 
such that $\Omega\cap\mathbb{R}=(E/2,2E)$, and $(E/2,2E)+\ri[-1,1]\subset\Omega$.
Moreover, let $m>0$ 
to be fixed later.
Let $I_1, \ldots, I_{N(h)}$ be a partition of $(E/2,2E)$ into intervals, i.e.,
\beq\label{eq:partition}
(E/2,2E)=\bigcup_{j=1\dots N(h)}I_{j},
\eeq
with $|I_{j}|=10\Cw h^{ m }$ for $j=1,\dots,N(h)-1$ and $|I_{N}|\leq10\Cw h^{ m }$, where $\Cw$ will be chosen later (the subscript ${\rm w}$ in $\Cw$ emphasises that this constant dictates the width of the intervals in the partition of $(E/2, 2E)$).
Let 
\beq\label{eq:J0def}
J'(h):=\bigcup_{(I_{j}+\ri[-1,1])\cap\mathcal{R}\neq\emptyset}I_{j}.
\eeq
The set $J'(h)$ can be written as a disjoint union
\beq\label{eq:aibi}
J'(h)=\bigcup\big[a_{i},b_{i}\big]\cap\Omega,
\eeq
(where the intersection with $\Omega$ is taken to ensure that $J'(h)\subset(E/2,2E)$, as implied by its definition \eqref{eq:J0def}). Let 
\beq\label{eq:Jdef}
J''(h):=\bigcup\big[a_{i}-3\Cw h^{ m },b_{i}+3\Cw h^{ m }\big].
\eeq
This set-up implies that every point of $(E/2, 2E)\backslash J''(h)$ has a neighbourhood
of the form 
\beqs
[w-2\Cw h^{ m },w+2\Cw h^{ m }]+\ri\big[-1/2,1/2\big]
\eeqs
that is disjoint from 
\[
\bigcup_{z\in\mathcal{R}}D(z, \Cw h^{ m }).
\]
Theorem \ref{thm:TZkey} therefore implies that in these neighbourhoods the semiclassical resolvent $R(w,h)$ satisfies
\begin{align*}
\Vert\chi R(w,h)\chi\Vert_{\cH\rightarrow \cH}&\leq C_1\exp\left(C_1 h^{-n^{\#}}\log\left(\frac{1}{\Cw h^{ m }}\right)\right),\\
&= C_1 \exp \left(C_1 h^{-n^{\#}}\left[\log \left(\frac{1}{\Cw}\right)+  m \log\left(\frac{1}{ h}\right)\right]\right),
\end{align*}
for all $0<h<h_0$, where $h_0$ and $C_1$ are given in \eqref{eq:TZkey1} and depend on $\Omega$, and hence on $E$. Therefore, given $\eta>0$, by choosing $h_1=h_1(h_0, \eta, \Cw,  m )$ sufficiently small, 
\beqs
\Vert\chi R(w,h)\chi\Vert_{\cH\rightarrow \cH}\leq C_1 \exp\left(C_1  m \, h^{-(n^{\#}+ \eta)}\right) \quad\tfa 0<h<h_1.
\eeqs

Since the resolvent also satisfies the bound \eqref{eq:absbound2}, we
can apply Theorem \ref{thm:scmp} (the semiclassical maximum principle)
with $Q= \chi R\chi /C_1$, $C= C_1 m$, $a(h)=\Cw h^{ m }$, $L=n^{\#}+\eta$ with
$\eta>0$ arbitrary small, and the largest possible $\delta(h)$
permitted by \eqref{eq:restrict1}, namely 
\beqs
\delta(h)=ch^{ m +3L/2}=ch^{ m +3n^{\#}/2+3\eta/2} 
\eeqs 
where $c>0$
is sufficiently small (depending on $m$ and $\Cw$); the result is that there exists a $C_2>0$
(depending on $C_1$, $ m $, and $\Cw$), such that \beq\label{eq:temp1}
\Vert\chi R(w,h)\chi\Vert_{\cH\rightarrow \cH}\leq C_2 h^{-( m
  +3n^{\#}/2+3\eta/2)}\,\, \tfa w\in(E/2,2E)\backslash
J''(h) 
\eeq and for all $0<h<h_1$. Observe that, at the price of making $C_2$
bigger, we can set $h_1=1$. More precisely, \eqref{eq:temp1} and the
fact that $\Vert\chi R(w,h)\chi\Vert_{\cH\rightarrow \cH}$ is bounded
for all $h>0$ imply that there exists $C_3>0$ (depending on $C_1, m, \Cw $,
and $h_1$, and thus on $C_1, m, \Cw, h_0,$ and $\eta$), such that
\beq\label{eq:temp1a} \Vert\chi R(w,h)\chi\Vert_{\cH\rightarrow
  \cH}\leq C_3 h^{-( m +3n^{\#}/2+3\eta/2)}\,\, \tfa
w\in(E/2,2E)\backslash J''(h) \eeq and for all $0<h\leq 1$.

We now need to estimate the size of $J''(h)$. For
$z\in(E/2,2E)+\ri[-1,1]$, $h^{-1}z^{1/2}$ is contained in a ball of
radius proportional to $h^{-1}$  in an angular sector with angle independent of $h$. Therefore, by the bound  \eqref{eq:count} on the number of resonances of $P$, there exists $C^{\#}>0$ such that
\beq\label{eq:count1}
{\rm card}(\Omega\cap\mathcal{R})\leq C^{\#}h^{-n^{\#}},
\eeq
and so we also have that 
$\text{card}\{j,\ (I_{j}+ \ri[-1,1])\cap\mathcal{R}\neq\emptyset\}\leq C^{\#} h^{-n^{\#}}$.
The measure of $J''(h)$ is bounded by the number of intervals in the definition \eqref{eq:Jdef} multiplied by the width of the intervals, and thus
\beq\label{eq:count2}
|J''(h)|\leq C^{\#} h^{-n^{\#}}\times 6 \Cw h^{ m }=6C^{\#}\Cw h^{ m -n^{\#}}.
\eeq

The plan for the rest of the proof is to obtain the bound \eqref{eq:thmbb11} on the non-semiclassical resolvent 
$\chi(P-k^{2})^{-1}\chi$ for $k\in [k_0,\infty)$ (i.e.~$k^2 \in [k_0^2, \infty)$) by taking $E=k_0 ^2$, writing 
\beqs
[k_0^2,\infty) = \bigcup_{\ell=0}^\infty \big[2^{\ell}E,2^{\ell+1}E\big),
\eeqs
applying the resolvent estimate \eqref{eq:temp1a} in each interval, choosing $ m $ so that the union of the excluded sets has finite measure, and finally choosing $\Cw $ so that this measure is bounded by $\delta$. Indeed, 
if $k^2\in[2^{\ell}E,2^{\ell+1}E)$,
then $2^{-\ell}k^2\in[E,2E)\subset (E/2, 2E)$. We now apply the estimate (\ref{eq:temp1})
with $h=2^{-\ell/2}$ and $w = h^{2}k^{2}$; observe that the smallest $\ell$, namely $\ell=0$, corresponds to $h=1$, i.e.~the largest $h$ for which the estimate \eqref{eq:temp1a} is valid. The result is that,
\begin{align}\nonumber
\Vert\chi(P-k^2)^{-1}\chi\Vert_{\cH\rightarrow \cH}=h^{2}\Vert\chi(h^{2}P-h^{2}k^2)^{-1}\chi\Vert_{\cH\rightarrow \cH}
&\leq C_3 
h^2 h^{-( m +3n^{\#}/2+3\eta/2)},\\ \nonumber
&\hspace{-0.5cm}\leq C_3 \left(\frac{k}{\sqrt{E}}\right)^{(-2+  m +3n^{\#}/2+3\eta/2)},\\
&\hspace{-0.5cm}\leq C k^{(-2+  m +3n^{\#}/2+3\eta/2)},\label{eq:temp2}
\end{align}
for all $h^2 k^2 \in (E/2,2E)\setminus J''(h)$, and in particular for
all $k^2 \in [2^{\ell}E,2^{\ell+1}E)\setminus \widetilde{J}_\ell$,
where 
\beqs 
\widetilde{J}_{\ell}:=\left\{z \in[2^{\ell}E,2^{\ell+1}E)
  : 2^{-\ell}z\in J''(2^{-\ell/2})\right\}.  
  \eeqs 
The bound
\eqref{eq:temp2} will become the bound \eqref{eq:thmbb11} in the
result (after $ m $ is specified). Observe that the constant $C$ in
\eqref{eq:temp2} depends on $C_3, E, m , n^{\#},$ and $\eta$; tracking
through the dependencies of $C_3$ (described above), and using the
fact that $E=k_0^2$, we find that $C$ depends on
$k_0, m , n^{\#}, \eta, \Cw, C_1,$ and $h_0$.

We then set 
\beq\label{eq:tJ}
\widetilde{J}:= \bigcup_{\ell=0}^\infty \widetilde{J}_\ell,
\eeq
so that the bound \eqref{eq:temp2} holds for $k^2 \in [k_0^2, \infty)\setminus \widetilde{J}$. 
We now choose  $ m $ so that $\widetilde{J}$ has finite measure; indeed, by \eqref{eq:count2},
\beq\label{eq:tJ2}
\big|\widetilde{J_\ell}\big|\leq 2^{\ell} 6 C^{\#}\Cw   \,2^{-\ell ( m -n^{\#})/2} = 6 C^{\#}\Cw  \,2^{-\ell(  m -n^{\#}-2)/2}.
\eeq
Taking
\beq\label{eq:wtm}
 m =n^{\#}+2+\widetilde{\eps},
\eeq
and using \eqref{eq:tJ} and \eqref{eq:tJ2} yields
\beq\label{eq:measureJ}
|\widetilde{J}| \leq 6 C^{\#}\Cw \sum_{\ell=0}^\infty 2^{-\ell \widetilde{\eps}/2}= 6 C^{\#}\Cw  \frac{1 }{1 - 2^{-\widetilde{\eps}/2}},
\eeq
and so $|\widetilde{J}|<\infty$ for every $\widetilde{\eps}>0$. We now use the freedom we have in choosing $\Cw$ to make $|\widetilde{J}|$ arbitrarily small: given $\delta'>0$ and $\widetilde{\eps}>0$, let 
\beqs
\Cw := \frac{\delta' (1 - 2^{-\widetilde{\eps}/2})}{6 C^{\#} },
\eeqs
so that $|\widetilde{J}|\leq \delta'$ by \eqref{eq:measureJ}. We now define $J$ so that 
\beq\label{eq:Jdef2}
k\in [k_0,\infty) \setminus J
\quad \text{ if and only if } \quad k^2 \in [k_0^2,\infty) \setminus \widetilde{J}.
\eeq
Since $|J|\leq |\widetilde{J}|/(2k_0)$ 
given $\delta>0$, let $\delta':= 2 \delta k_0$, so that $|J|\leq \delta$. 
We have therefore proved that the bound \eqref{eq:temp2} holds with $ m $ given by \eqref{eq:wtm} for all $k\in [k_0,\infty)\setminus J$. 
The bound \eqref{eq:thmbb11} then follows from \eqref{eq:temp2} with $\eps:=  3 \eta/2 + \widetilde{\eps}$. 
The constant $C$ in \eqref{eq:thmbb11} depends on $k_0, n^{\#}, \delta, \eps, C^{\#}, C_1,$ and $h_0$, where $C_1$ and $h_0$ are defined in Theorem \ref{thm:TZkey} and depend on $k_0$, and $C^{\#}$ is defined in \eqref{eq:count1} and arises from the bound \eqref{eq:count} on the number of resonances.
\end{proof}

\bre[Multiplicities]\label{rem:multiplicity}
In \eqref{eq:count1}
we are concerned with the
\emph{distinct locations}
of
resonances in $\Omega$, while the bound \eqref{eq:thmbb11} is 
unaffected by their multiplicity.  If we assume that the multiplicity of all but finitely
many resonances is proportional to $k^{\rho},$ the number of distinct locations is reduced, and the bound \eqref{eq:count1} is
replaced by
${\rm card}(\Omega\cap\mathcal{R})\lesssim h^{-n^{\#}+\rho}$; one can
then take $ m = n^{\#}+2 + \widetilde{\eps} - \rho$, and the bound
\eqref{eq:thmbb11} is improved by a factor of $k^{-\rho}$.  
Although
such multiplicity assumptions are highly nongeneric--see, e.g.,
\cite[Theorem 4.39]{DyZw:19}--  they hold, however, in certain
situations with a high degree of symmetry;
see Corollary \ref{cor:transmission} below for an example.
\ere

\begin{theorem}[Improvement of Theorem \ref{thm:bb1} under stronger assumption on location of resonances]\label{thm:bb2}
Assume that, given $c_j>0$, $j=1,2$, the number of resonances of $P$
in the box 
\beq\label{eq:smallerbox}
[r,r+c_1 r^{-1}]+\ri[-c_2,c_2]
\eeq
is $\lesssim r^p$ for some $p>0$
and for all $r>0$. 

(i) Given $\eps>0$, $\delta>0$, and 
\beq\label{eq:choiceofA}
s \geq - \frac{1}{2} + \frac{n^\#}{2} + \frac{4 \eps}{5},
\eeq
there exists $\lambda_0 = \lambda_0(\eps,\delta, s, n^\#)>1$ such that
\begin{align}
\Big|\big\{w, \Vert\chi(P-w)^{-1}\chi\Vert_{\cH\rightarrow \cH}>\lambda^s\big\} \cap[\la,\la+1]\Big|
\leq \delta {\la}^{ -s -1 + 3n^{\#}/4 + p/2 + \eps}\label{eq:mu1}
\end{align}
for all $\la>\la_0$.

(ii) Given $k_0>0$, $\delta>0$, and $\eps>0$, there exists a constant $C(k_0, \delta,\eps, n^{\#}) >0$
and a set $J$ with $|J|\leq \delta$ such that the resolvent \eqref{eq:resolvent1} satisfies
\beqs
\Vert\chi R(k)\chi\Vert_{\cH\rightarrow \cH}\leq C
k^{3n^{\#}/2 +p + \eps}
\quad \tfa k \in [k_0,\infty)\backslash J.
\eeqs

\end{theorem}

Before proving Theorem \ref{thm:bb2}, we note that two situations where the hypotheses of Theorem \ref{thm:bb2} on the number of resonances can be verified are: (a) where one has a sharp Weyl remainder in the asymptotics of the eigenvalue counting function for the reference operator $P^{\#}$ -- see Corollaries \ref{cor:bb2} and \ref{cor:bb3} below, and (b) where one has Weyl-type asymptotics for the counting function of the resonances of $P$ -- see Corollary \ref{cor:transmission} below  for the specific case of a penetrable obstacle. 
We highlight that the hypotheses can be verified in (a) thanks to the results of \cite{PeZw:99}, \cite{Bo:01}, and \cite{SjZw:07} (see Corollary \ref{cor:bb2} below for more detail). We also note that, in both cases (a) and (b), the number of resonances in $[r,r+c_1 r^{-1}]+\ri[-c_2,c_2]$ is estimated by the number in $[r,r+c_1]+\ri[-c_2,c_2]$; although we see below that the former set arises naturally in the proof of Theorem \ref{thm:bb2}, rigorous results about resonance distribution on the $r^{-1}$ scale seem well out of reach of current methods.

\begin{proof}[Proof of Theorem \ref{thm:bb2}]
\emph{Proof of Part (i)}. We argue as in Theorem \ref{thm:bb1} except that now we work in an interval of size $h^2$ instead of $3E/2$ and choose the intervals comprising $J'$ to have smaller imaginary part. Indeed, 
let $\Omega\Subset\left\{ \Real\,z>0\right\} $ be such that $\Omega\cap\mathbb{R}=(1, 1+h^2)$, and $(1, 1 +h^2)+\ri[-h,h]\subset\Omega$.
Let 
$I_1, \ldots, I_{N(h)}$ be a partition of $(1, 1 +h^2 )$ into intervals, i.e.,
\beqs
(1, 1 +h^2)=\bigcup_{j=1\dots N(h)}I_{j},
\eeqs
(compare to \eqref{eq:partition})  with $|I_{j}|=10\Cw h^{ m}$ for $j=1,\dots,N(h)-1$ and $|I_{N}|\leq10\Cw h^{ m}$, where $m>0$ and $\Cw>0$ will be chosen later.
Let 
\beqs
J'(h):=\bigcup_{(I_{j}+\ri[-h,h])\cap\mathcal{R}\neq\emptyset}I_{j}.
\eeqs
With $J'(h)$ written as \eqref{eq:aibi}, let $J''(h)$ to be defined by
\eqref{eq:Jdef}. As in the proof of Theorem \ref{thm:bb1}, every point of $(\Omega\cap\mathbb{R})\backslash J''(h)$ has a neighbourhood of the form 
\beqs
[w-2\Cw h^{ m},w+2\Cw h^{ m}]+\ri\big [-h/2,h/2],
\eeqs
that is disjoint from 
\beqs
\bigcup_{z\in\mathcal{R}}D(z,\Cw h^{ m}),
\eeqs
and thus where Theorem \ref{thm:TZkey}  implies that the semiclassical resolvent $R(w,h)$ satisfies
\beqs
\Vert\chi R(w,h)\chi\Vert_{\cH\rightarrow \cH}\leq C_1 \exp \left(C_1 h^{-n^{\#}}\left[\log \left(\frac{1}{\Cw}\right)+  m\log\left(\frac{1}{ h}\right)\right]\right),
\eeqs
for all $0<h<h_0$, where $h_0$ and $C_1$ are given in \eqref{eq:TZkey1} and depend on $\Omega$.
Arguing as in the proof of Theorem \ref{thm:bb1}, we find that, given $\eta>0$, by choosing $h_1=h_1(h_0, \eta, \Cw,  m )$ sufficiently small, 
\beqs
\Vert\chi R(w,h)\chi\Vert_{\cH\rightarrow \cH}\leq C_1 \exp\left(C_1  m\, h^{-(n^{\#}+ \eta)}\right) \quad\tfa 0<h\leq h_1.
\eeqs
We now use
the semiclassical maximum principle, Theorem \ref{thm:scmp},
with $Q= \chi R \chi/C_1$, $a(h)= \Cw h^{ m }$,
$L=n^{\#}+\eta$ with $\eta>0$ arbitrary small, $C=C_1  m$ and the largest possible $\delta(h)$ permitted by \eqref{eq:restrict1}, namely
\beq\label{eq:deltah2}
\delta(h)=c h^{ m +3L/2} = c h^{ m +3n^{\#}/2+3\eta/2}.
\eeq
where $c\leq \Cw(C_1 m)^{-1/2}$.
Note that, to apply the semiclassical maximum principle, we need  $(-\delta(h)h^{-L},\delta(h)) \subset (-h/2,h/2)$. Therefore, we assume, and check later, that with our choice of $c$ and $ m$,
\beq \label{eq:cond_delt}
\delta(h)\leq \frac{1}{2}h^{1+L}.
\eeq
The result is that there exists $C_2>0$ such that
\beq\label{eq:temp1aa}
\Vert\chi R(w,h)\chi\Vert_{\cH\rightarrow \cH}\leq C_2 h^{-( m +3n^{\#}/2+3\eta/2)}\, \tfa w\in(1,1+h^{2})\backslash J''(h),
\eeq
and for all $0<h\leq h_1$,
Just as in the proof of Theorem \ref{thm:bb1}, at the price of making $C_2$ bigger, we can assume that $h_1=1$.
Observe that, by choosing $c$ sufficiently small in the definition of $\delta(h)$ \eqref{eq:deltah2}, the condition \eqref{eq:cond_delt} is satisfied when
\beq\label{eq:boundonm}
h^{ m+3n^{\#}/2+3\eta/2-1} \lesssim h^L
\eeq
for $h$ sufficiently small. 

As in Theorem \ref{thm:bb1}, we bound $|J''(h)|$ by the number of intervals multiplied by their widths. As before, the widths are bounded by $6\Cw h^{ m }$, but now the number of  intervals -- corresponding to the number of semiclassical resonances in $[1,1+h^2]+\ri[-h,h]$ -- is bounded by  $C_1^{\#} h^{-p}$, where $C_1^{\#}$  depends only on $P$. Indeed, by Lemma \ref{lem:box}, 
the image of the box  $[1,1+h^2]+\ri[-h,h]$ under the scaling $z\rightarrow h^{-1} z^{1/2}= k$  is included in a box of form $[h^{-1},h^{-1}(1+c_1 h^2)]+\ri[-c_2,c_2]$ for some $c_j>0, j=1,2,$ independent of $h$, and by the assumption in the theorem, the number of resonances of $P$ in this latter box is bounded, up to a multiplicative constant which we denote by $C_1^{\#}$, by $h^{-p}$. Therefore, 
\beq \label{eq:measureJsec}
|J''(h)|\leq C_1^{\#} h^{-p} \times 6\Cw h^{ m}=6 \Cw C_1^{\#} h^{ m-p}.
\eeq

Having obtained the bound \eqref{eq:temp1aa}, we now seek an upper bound on the measure of the set where $\Vert\chi R(w,h)\chi\Vert_{\cH\rightarrow \cH}> h^{-t} =: B(h)$. The choice of $t$ here will dictate our choice of $ m $ (and hence the measure of the set via \eqref{eq:measureJsec}).
Observe that
\beqs
C_2 h^{-( m +3n^{\#}/2+3\eta/2)}\leq  h^{-t}
\eeqs
if and only if
\beq
m\leq t + \frac{\log (1/C_2)}{\log(1/h)}-3n^{\#}/2-3\eta/2.
\label{eq:Wineq0}
\eeq
Since $C_2$ is independent of $h$, there exists an $h_2>0$ such that the inequality \eqref{eq:Wineq0} holds when 
\beq\label{eq:new_choice_of_m}
m= t- \eta - 3n^{\#}/2-3\eta/2
\eeq
and $0<h\leq h_2$. Note that $h_2$  depends on $C_2$ and on the choice of $m$, and hence on $n^{\#}, \eta,$ and $C_\text{w}$.

Observe that with the choice of $m$ \eqref{eq:new_choice_of_m}, we see that the inequality \eqref{eq:boundonm} holds, in particular, when
\beq \label{eq:choiceofB}
t\geq 1+ L + \eta. 
\eeq

We now input the information about $ m$ into our bound on the measure of the set $J''(h)$. Indeed, 
from \eqref{eq:temp1aa} and our choice of $m$ \eqref{eq:new_choice_of_m}, for $0<h\leq h_2$ and $w\in (1,1+h^2)$,
\[
\Vert\chi R(w,h)\chi\Vert_{\cH\rightarrow \cH}>B(h)\quad \text{ implies that }\quad  w\in J''(h).
\]
Therefore, choosing $\Cw $ small enough so that
$
6 \Cw C_1^{\#} \leq \delta,
$
we get, by \eqref{eq:measureJsec}, for $0<h\leq h_2$,
\begin{align}
\Big|\big\{ \Vert\chi R(w,h)\chi\Vert_{\cH\rightarrow \cH}>B(h)\big\} \cap[1,1+h^{2}]\Big|&\leq|J''(h)|  \label{eq:temp4} \\
&\leq 6 \Cw C_1^{\#} h^{ m-p} \leq  \delta  h^{t - 3 n^\#/2 - p - 5 \eta/2}.\nonumber
\end{align}
Since
\begin{align*}
&\begin{cases}
w\in[\lambda,\lambda+1],\\
\Vert\chi(P-w)^{-1}\chi\Vert_{\cH\rightarrow \cH}>A(\la)
\end{cases}
\\
&\qquad\text{ if and only if }\quad
\begin{cases}
h^{2}w\in[1,1+h^{2}], \text{ with } h=\la^{-1/2},\\
\Vert\chi R(h^{2}w,h)\chi\Vert_{\cH\rightarrow \cH}>B(h), \text{ with } B(h)=h^{-2}A(h^{-2}),
\end{cases}
\end{align*}
applying this with $A(\la)= \la^s$ and hence $B(h) = h^{-2s -2}$ i.e. $t=2s+2$, and using the bound \eqref{eq:temp4}, we have that, for $\la \geq h_2 ^ {-2}$
\begin{align*}
\Big|\big\{ \Vert\chi(P-w)^{-1}\chi\Vert_{\cH\rightarrow \cH}>A(\lambda)\big\} \cap[\lambda,\lambda+1]\Big|
\leq \delta {\la}^
{-s -1 +3 n^\#/4 + p/2 + 5 \eta/4}.
\end{align*}
This last bound implies the result \eqref{eq:mu1} with $\eps=5\eta/4$ and $\la_0 =  h_2 ^ {-2}$. Recalling that $h= \la^{-1/2}$, one can check that the 
condition \eqref{eq:choiceofB} is satisfied by the hypothesis \eqref{eq:choiceofA}. 

\emph{Proof of Part (ii).}
%
First of all, observe that it is sufficient to prove that there exists $J \subset [k_1,\infty)$ with $|J| \leq \delta$ such that
\beq\label{eq:corbb1_red}
\Vert\chi R(k)\chi\Vert_{\cH\rightarrow \cH}\leq C
k^{3n^{\#}/2 +p + \eps}
\quad \tfa k \in [k_1,\infty)\backslash J,
\eeq
where $k_1 > k_0$. Indeed, if \eqref{eq:corbb1_red} holds, the result follows by increasing the constant $C$ so that the estimate still holds in $[k_0,\infty)\backslash J.$
We therefore now prove \eqref{eq:corbb1_red}.

Let $\delta_0>0$ be a constant to be fixed later, and 
\beqs
s := 3n^{\#}/4 +p/2 + 2\eps;
\eeqs
observe that this choice satisfies the requirement \eqref{eq:choiceofA}. Now, let $\la_0 = \la_0 (\delta_0, s, 2 \eps)$ be given by Theorem \ref{thm:bb2}.
We set
\beq\label{eq:temp5}
\widetilde{J} := \Big\{\Vert\chi(P-w)^{-1}\chi\Vert_{\cH\rightarrow \cH}> w^s\Big\} \cap [\lambda_0,+\infty),
\eeq
so that 
\beq\label{eq:explain}
\N{\chi(P-\la)^{-1}\chi}_{\cH\rightarrow \cH}\leq 
\lambda^{3n^{\#}/4 +p/2 + 2\eps}
\quad\tfa \lambda \in [\lambda_0,+\infty) \backslash \widetilde{J}.
\eeq
We now bound the measure of $\widetilde{J}$ using Theorem \ref{thm:bb2}. Indeed, 
by Theorem \ref{thm:bb2}, for all $\la \geq \la_0,$
\beq\label{eq:utth}
\Big|\big\{w, \Vert\chi(P-w)^{-1}\chi\Vert_{\cH\rightarrow \cH}>\la^{s}
)\big\} \cap[\la,\la+1]\Big|\leq \delta_0 {\la}^{-1- \eps}.
\eeq
From the definition of $\widetilde{J}$ \eqref{eq:temp5},
\begin{multline}\nonumber
|\widetilde{J}| = \sum_{k\geq0,\text{ }\lambda=\lambda_0+k} \Big|\big\{w, \Vert\chi(P-w)^{-1}\chi\Vert_{\cH\rightarrow \cH}> w^s \big\} \cap[\la,\la+1]\Big| \\
\leq \sum_{k\geq0,\text{ }\lambda=\lambda_0+k}\Big|\big\{w, \Vert\chi(P-w)^{-1}\chi\Vert_{\cH\rightarrow \cH}> \la^s \big\} \cap[\la,\la+1]\Big|,
\end{multline}
where this last inequality holds because the function $w \mapsto w^s $ is increasing.
Therefore, by \eqref{eq:utth}
\beqs
|\widetilde{J}| \leq  \delta_0 \sum_{k\geq0,\text{ }\lambda=\lambda_0+k} {\la}^{-1-\eps} \leq  \delta_0 \int_{\la_0}^{\infty} \la^{-1- \eps}  d\la\, = \delta_0 \frac{1}{\eps} (\la_0)^{-\eps}\leq \frac{\delta_0}{\eps},
\eeqs
thus, choosing $\delta_0 := 2\delta \eps$,
the estimate \eqref{eq:corbb1_red} follows from \eqref{eq:explain} with $\la = k^2$, $k_1 ^ 2 = \la_0$, and $J$ defined by \eqref{eq:Jdef2}. Observe that, since $k_1>1$, arguing in a similar way to the proof of Theorem \ref{thm:bb1} (in the text after \eqref{eq:Jdef2}), we have that $|J|\leq |\widetilde{J}|/2 \leq \delta$.
\end{proof}

\begin{corollary}[Improved resolvent estimate under   sharp Weyl remainder for reference operator.]
  \label{cor:bb2}
Let the assumption \eqref{eq:gro} on the growth of the
spectral counting function for the black-box reference operator be
replaced by the stronger assumption
\beq\label{eq:hormanderweyl}
N(P^{\#}, (-C,\lambda]) = C \lambda^{n/2} + O(\lambda^{(n-1)/2}) \quad\tas \lambda\tendi.
\eeq
Then, given $k_0>0$, $\delta>0$, and $\eps>0$, there exists a constant $C(k_0, \delta,\eps, n^{\#}) >0$
and a set $J$ with $|J|\leq \delta$ such that the resolvent \eqref{eq:resolvent1} satisfies
\beqs
\Vert\chi R(k)\chi\Vert_{\cH\rightarrow \cH}\leq C
k^{5n/2 -1 + \eps}
\quad \tfa k \in [k_0,\infty)\backslash J.
\eeqs
  \end{corollary}

\bpf
This follows from Theorem~\ref{thm:bb2}
using the  result of Petkov--Zworski \cite[Proposition 2]{PeZw:99} that, under the Weyl-law assumption on the reference operator
  \eqref{eq:hormanderweyl}, the number of resonances in
$[r,r+c_1 ]+\ri[-c_2,c_2]$ (and hence also in the smaller box \eqref{eq:smallerbox}) is bounded by $C_1 r^{n-1}$ for some $C_1>0$; i.e.~the assumptions of Theorem \ref{thm:bb2} are satisfied with $p=n-1.$ (See also  \cite[Theorem 1]{Bo:01} and \cite[Theorem 2]{SjZw:07} for later refinements on \cite{PeZw:99}.)
\epf

A particularly-important situation where the assumptions of Corollary \ref{cor:bb2} apply is scattering by Dirichlet or Neumann obstacles with $C^{1,\sigma}$ boundaries.
  
\begin{corollary}[Improved resolvent estimate for scattering by $C^{1,\sigma} $ Dirichlet or Neumann obstacles]\label{cor:bb3}
Let $\obstacle_-, \obstacle_+,$ and $A$ be as in Lemma \ref{lem:obstacle}, and assume further that both $A$ and $\partial\obstacle_+$ are $C^{1,\sigma}$ for some $0<\sigma<1$ (observe that this also includes the case when $\obstacle_-=\emptyset$).
Then, given $k_0>0$, $\delta>0$, and $\eps>0$, there exists $C=C(k_0, \delta, \eps, n)>0$ and a set $J\subset [k_0,\infty)$ with $|J|\leq \delta$ such that
\beqs
\N{\chi R(k)\chi}_{\LtLt}\leq Ck^{5n/2-1+ \eps}
\quad \tfa k \in [k_0,\infty)\backslash J.
\eeqs
\end{corollary}

\bpf
The result that the asymptotics \eqref{eq:hormanderweyl} hold when, additionally, $A$ and $\partial\obstacle_-$ are smooth goes back to Seeley \cite{Se:80} and Ph\d{a}m The L\d{a}i \cite{Ph:81}. The more-recent results of Ivrii \cite{Iv:00, Iv:03, Iv:13} extend this result to much more general classes of coefficients and domains, including those that are $C^{1,\sigma}$ for some $0<\sigma<1$ \cite{Iv:00}.
\epf

\begin{corollary}[Improved resolvent estimate for scattering by a 3-d penetrable ball]\label{cor:transmission}
Let $R(k)$ be the resolvent in the case of scattering by a penetrable obstacle (described in Lemma \ref{lem:transmission} and Remark \ref{rem:transmission}) when, furthermore, the obstacle $\obstacle_-$ is a 3-d ball and $c<1$ so that the problem is trapping (see Remark \ref{rem:trappingtransmission}). 
Assume that the parameter $\alpha$ in the transmission condition \eqref{eq:transmission2} satisfies $\alpha\leq \alpha_0$, where $\alpha_0>0$ is as in \cite[Theorem 1.1]{CaPoVo:01}.
Then, given $k_0\geq1$, $\delta>0$, and $\eps>0$, there exists a constant $C(k_0,\delta, \eps) >0$
and a set $J$ with $|J|\leq \delta$ such that the resolvent \eqref{eq:resolvent1} satisfies
\beq\label{eq:cor_transmission}
\Vert\chi R(k)\chi\Vert_{\cH\rightarrow \cH}\leq C
k^{6+ 1/6 + \eps}
\quad \tfa k \in [k_0,\infty)\backslash J.
\eeq
\end{corollary}

The exponent $6+ 1/6$ in \eqref{eq:cor_transmission} should be compared to the exponent $7+1/2$ from Theorem \ref{thm:bb1} (recall that $n=3$ here).

\bpf[Proof of Corollary \ref{cor:transmission}]
Let $N(r)$ denote the number of resonances in $[0,r]+\ri [-c_2,c_2]$.
By \cite[Theorem 1.3]{CaPoVo:01}, there exists $C_1>0$ such that, given
$\epsilon>0$, 
\beqs N(r) = C_1 r^n + O_\epsilon(r^{n-1/3+\epsilon})
\quad\tas r\tendi,
\eeqs
where $n=3$.
Then
 \begin{align*}
 N\big(r+r^{-1}\big)-N(r) &= C_1 \big( (r+r^{-1})^n - r^n \big) + O\big( (r+r^{-1})^{n-1/3+\epsilon}\big) + O\big(r^{n-1/3+\epsilon}\big)\\ &= O( r^{n-1/3+\epsilon}),
 \end{align*}
 and the assumptions of Theorem~\ref{thm:bb2} hold with $p= n-1/3+\epsilon = 3-1/3 +\epsilon$ (note that this application makes no use of the fact that the interval
 $[r,r+r^{-1}]$ is shrinking as $r \to\infty$ rather than having fixed
 width, i.e., $N(r+1)-N(r)$ enjoys the same estimate).
 
The bound \eqref{eq:cor_transmission} then follows from Remark
\ref{rem:multiplicity} if we can show that all but finitely-many
resonances have multiplicity proportional to $k$. Indeed, assuming this multiplicity property,
in the proof of Theorem \ref{thm:bb2}, instead of the number of semiclassical resonances in $[1,1+h^2]+\ri[-h,h]$  being bounded by  $C_1^{\#} h^{-p}$, it is bounded by $C_1^{\#} h^{-p+1}$; this factor of $h=k^{-1}$ then propagates through the proof of Theorem \ref{thm:bb2}.

To prove this multiplicity property, we first recall that, when $c<1$ and the problem is trapping, the resonances fall into two groups by \cite[\S9, Page 137]{St:06}:
\ben
\item one near the resonances of the exterior Dirichlet problem for the ball --  since this latter problem is nontrapping, these resonances lie away from the real axis -- and
\item one near the real axis, with asymptotics given by 
\beq\label{eq:asym}
\frac{1}{c} k_{\nu, i} = \nu + \alpha_i \left(\frac{\nu}{2}\right)^{1/3} + O(1) \quad\tas \nu \tendi,
\eeq
where $\alpha_i$ denotes the $m$th zero of the Airy function ${\rm Ai}(-z)$ and $\nu:= \ell+1/2$, where $\ell$ is the angular frequency; see, e.g., 
\cite[Eq.~1.1]{LaLeYo:92}, \cite{BaDaMo:19}.
\een
Each resonance has multiplicity $2\ell+1$ because, by separation of variables, the solution can be expressed in the form
\beqs
\sum_{\ell=0}^\infty \sum_{m=-\ell}^{\ell} a_\ell(kr) Y_{\ell,m}(\theta,\phi),
\eeqs
where $a_\ell(\cdot)$ is either a spherical Hankel or spherical Bessel function (defined by \cite[\S10.47]{Di:19}) and $Y_{\ell,m}(\cdot,\cdot)$ are spherical harmonics (defined by \cite[Eq.~14.30.1]{Di:19}); see, e.g., \cite[Eq.~3.1-3.3]{CaLePa:12}. 
By \eqref{eq:asym} and the fact that $\nu:=\ell+1/2$, the multiplicity of each resonance is proportional to $k$, and the proof is complete.
\epf

The final result of this section (Lemma \ref{lem:lowerbd}) is a lower bound on the resolvent for
all frequencies in an ``equidistribution of resonances" scenario. 
In fact, it is more convenient to
  work with \emph{quasimodes} (sequences of approximate solutions to
  the Helmholtz equation with real spectral parameter) rather than
  resonances, since the existence of quasimodes is usually easier to
  establish in cases of stable trapping, and in many cases is known to
  be equivalent to the existence of sequences of resonances
  approaching the real axis; see \cite{StVo:95}, \cite{StVo:96},
  \cite{TaZw:98}, \cite{St:00}, \cite[\S7.3]{DyZw:19}. 

\begin{lemma}[Lower bound on resolvent under ``equidistribution of resonances" scenario]
\label{lem:lowerbd}
Assume that there exist a compact subset $K$ and $s \geq 0$ such that, for
all $k>0$ and for all $\lambda\in[k,k+1]$, there exists $C_1>0$, $\mu\in B(\lambda,C_1k^{-s})$,
and a $\mu$-quasimode for $P$, denoted by $u$, supported in $K$ and of
order $s-1$, i.e.
\[
\Vert(P-\mu^2)(u)\Vert_{\cH}=O(\mu^{-(s-1)}), \quad\text{ with } \, \Vert u\Vert_{\cH}=1.
\]
Then there exists a $C_2>0$ and a $\chi\in C_{c}^{\infty}$
such that the lower-bound 
\beq\label{eq:lower_bound}
\Vert\chi R(k)\chi\Vert_{\cH\rightarrow \cH}\geq C_2k^{s-1}
\eeq
holds for all $k>0$.
\end{lemma}

If
the two-term Weyl-type asymptotics,
  $N(r) = C_1 r^n + C_2 r^{n-1} +o(r^{n-1})\tas r\tendi$, hold, then,
  arguing as in the proof of Corollary \ref{cor:transmission}, the number of resonances in
$[k,k+1]$ is comparable to $k^{n-1}$. The case  $s=n-1$ in Lemma \ref{lem:lowerbd} therefore assumes that quasimodes corresponding to these resonances are spread out
evenly throughout this interval. The existence of many quasimodes
  is relatively easy to arrange
  (e.g.\ for a Helmholtz resonator), unfortunately the
  \emph{equidistribution} of these quasimodes' spectral parameters, while highly plausible, seems very
  difficult to verify.

\begin{proof}[Proof of Lemma \ref{lem:lowerbd}]
Let $\lambda\in[k,k+1]$, $u$, and $\mu$ be as above.
Then
  \beqs
(P-k^{2})u=(P-\mu^2)u+(k^{2}-\mu^2)u=:f =O_{\cH}(k^{-(s-1)}),
\eeqs
with $f$ having support in $K$ as well.   Thus, with $\chi$
compactly supported and equal to $1$ on $K,$ $u=
\chi u$ and $f=\chi f,$ so in particular, $(P-k^{2})(u)=\chi f.$
Since $u$ is certainly outgoing (because it has compact support),
$u=R(k) \chi f,$ i.e.,
$$
u=\chi R(k)\chi f,
$$
and this proves the lower bound.
\end{proof}

\bre[Comparison with the results of \cite{Ca:12, CaLePa:12}]\label{rem:Cap}
As noted in \S\ref{sec:main}, in the case of scattering by a 2- or 3-d penetrable ball \cite[Lemma 6.2]{Ca:12} and \cite[Lemma 3.6]{CaLePa:12} show that, for $k$ outside a set of small measure, the scattered field everywhere outside the obstacle is 
bounded in terms of the incident field with a loss of $2 +\alpha$ derivatives, with $\alpha>0$ arbitrary.
The nontrapping resolvent estimate \eqref{eq:nt_estimate} (which holds for the transmission problem when $c>1$ by \cite[Theorem 1.1]{CaPoVo:99}; see also \cite[Theorem 3.1]{MoSp:19}) can be used to show that
$\big\|u^S\big\|_{L^2} \lesssim  \big\|u^I\big\|_{L^2}$; see, e.g., \cite[Lemma 6.5]{LaSpWu:19a}.
With each derivative corresponding to a power of $k$, the results of
\cite{Ca:12} and \cite{CaLePa:12} therefore indicate a loss of
$k^{2+\alpha}$ over the non-trapping estimate (compare to the loss of
$k$ when $s=n-1$ and $n=2$ in \eqref{eq:lower_bound}). The lowest loss over the
nontrapping resolvent estimate we can prove is a loss of
$5+2/3+\eps~(= 1+ 5\times2/2 -1/3 +\eps+ \epsilon)$ from 
Theorem \ref{thm:bb2}
 with $n=2$ and $p= n-1/3 +\epsilon$ by the
results in \cite{CaPoVo:01} used in the proof of Corollary \ref{cor:transmission}. However (as highlighted above) our results hold in
much more general settings, not least scattering by a smooth obstacle
with strictly positive curvature that is not a ball, whereas the
results of \cite{Ca:12}, \cite{CaLePa:12} use the explicit expression
for the solution when the obstacle is a ball and so are restricted to
this setting.  \ere

\section*{Acknowledgements}

The authors thank 
Alex Barnett (Flatiron Institute), Yves Capdeboscq (University of Oxford), 
Kirill Cherednichenko (University of Bath), 
Jeff Galkowski (Northeastern University), Leslie Greengard (New York University), Daan Huybrechs (KU Leuven), Steffen Marburg (TU M\"unich), Jeremy Marzuola (University of North Carolina at Chapel Hill), Andrea Moiola (Universit\`a degli studi di Pavia), Mike O'Neil (New York University), 
Zo\"is Moitier (Universit\'e de Rennes 1), Stephen Shipman (Louisiana State University), and Maciej Zworski (University of California, Berkeley)
 for useful discussions.
We thank the referee for constructive comments that improved the presentation of the paper, and also for bringing to our attention the results of \cite{PeZw:99, Bo:01, SjZw:07}.
 DL and EAS acknowledge support from EPSRC grant EP/1025995/1.  JW
acknowledges partial support from NSF grant DMS--1600023.

\begin{appendix}

\section{Images of boxes under semiclassical scaling}

\

\begin{lemma}[Images of boxes in $\Com$ under semiclassical scaling]
\label{lem:box}
Given $0<h_0<1$, let $0<h\leq h_0$. The image of the box  $[1,1+h^2]+\ri[-h,h]$ under the mapping $z\rightarrow h^{-1} z^{1/2}$  is included in a box of form 
$[h^{-1},h^{-1}(1+c_1 h^2)]+\ri[-c_2,c_2]$ for some $c_1, c_2>0$ dependent on $h_0$ but independent of $h$.
\end{lemma}

\begin{proof}
Let 
$$
a := 1 + \ri h, \quad b := 1+h^2 + \ri h
$$
(so that the box $[1,1+h^2]+\ri[-h,h]$ has vertices $a,b, \overline{a},$ and $\overline{b}$),
and let $A$ and $B$ be the images of $a$ and $b$ under the mapping $z\rightarrow h^{-1} z^{1/2}$. 
One can check that $\Imag B <\Imag A$ for all $0<h<1$. Let
$$
l_-(h) := h^{-1},\quad  l_+(h) := \Real B, \quad \tand\quad p(h) := \Imag A, 
$$
and then 
the image of $[1,1+h^2]+\ri[-h,h]$ is included in $[ l_-(h), l_+ (h)]+ \ri [p(h), - p(h)]$.
Now, with $\theta(h) := \arg a$, and $\psi(h):= \arg b$, we have
\beqs
\theta(h)  = \arctan h = h + O(h^3) \tas {h \tendo}, 
\eeqs
and
\beqs
\psi(h) = \arctan\left(\frac{h}{1+h^2}\right) = h +O(h^3) \tas {h \tendo},
\eeqs
hence
\begin{align*}
 l_+ (h) = h^{-1} |b|^{1/2} \cos \big(\psi(h)/2\big) &= h^{-1} \big((1+h^2)^2+h^2\big)^{1/4} \big(1 - h^2/8 + O(h^3)\big) \\
 &= h^{-1}\big(1+5 h^2 /8 + O(h^3)\big),
\end{align*}
and
\begin{align*}
p(h) = h^{-1} |a|^{1/2} \sin \big(\phi(h)/2\big) &= h^{-1} (1+h^2)^{1/4} (h+ O(h^3)) \\
& = 1+ O(h^2),
\end{align*}
and the result follows with $c_1>5/8$ and $c_2>1$.
\end{proof}

\section{Weyl-type upper bound for reference operator for penetrable- and impenetrable-obstacle problems}\label{app:Weyl}

The aim of this Appendix is to show that the reference operator $P^{\#}$ associated with 
\emph{either} the impenetrable obstacle problem of Lemma \ref{lem:obstacle} \emph{or} the transmission problem of Lemma \ref{lem:transmission} satisfies
the Weyl-type upper bound (\ref{eq:gro}).

\begin{lemma}\label{lem:Weyl1}
The reference operator $P^{\#}$ associated to either the Dirichlet or the Neumann obstacle problems of Lemma \ref{lem:obstacle} (in particular with
$A$ Lipschitz) satisfies the Weyl-type upper bound
\beq\label{eq:appWeyl}
N\big(P^\#,[-C,\lambda]\big)\lesssim\lambda^{d/2}.
\eeq
\end{lemma}
\begin{proof}
We use the results of \cite{No:97} on heat-kernel asymptotics in Lipschitz Riemannian manifolds. Indeed, taking the measure density to be one, 
\cite{No:97} covers both Dirichlet and Neumann realisations of the divergence form operator $\nabla\cdot(A\nabla \cdot)$ with $A$ Lipschitz.
By \cite[Theorem 1.1]{No:97}, the heat-kernel (for either Dirichlet or Neumann boundary conditions), $p(t,x,y)$, satisfies
\[
t\log p(t,x,y)\longrightarrow-\frac{1}{4}d(x,y)^{2} \quad \tas t\rightarrow 0,
\]
and therefore, in particular,
\begin{equation}
p(t,x,y)\lesssim\exp\left(-\frac{1}{4t}d(x,y)^{2}\right).\label{eq:heat_est}
\end{equation}
Recall, however, that
\[
p(t,x,y)=\sum_{n\in\mathbb{N}}\phi_{n}(x)\phi_{n}(y)\exp(-t\lambda_{n}),
\]
where $\phi_{n}$ is the eigenfunction of $L^2$-norm one associated with $\lambda_{n}$. Therefore, taking the square of (\ref{eq:heat_est})
and integrating with respect to $x$ and $y$ we obtain, by orthogonality
of the eigenfunctions,
\[
\sum_{n\in\mathbb{N}}\exp(-2t\lambda_{n})\lesssim\int\int\exp\left(-\frac{1}{2t}d(x,y)^{2}\right)\rd x \rd y=Ct^{-d/2};
\]
the result \eqref{eq:appWeyl} follows by a weak version of the Karamata Tauberian theorem appearing in, e.g., \cite[Proposition B.0.12]{Sp:08}.
\end{proof}

\begin{lemma}\label{lem:Weyl2}
The reference operator $P^{\#}$ associated to the transmission problem of Lemma \ref{lem:transmission} (in particular with
$A$ Lipschitz) satisfies the Weyl-type upper bound \eqref{eq:appWeyl}.
\end{lemma}
\begin{proof}
By the min-max principle for self-adjoint operators with
compact resolvent (see, e.g., \cite[Page 76, Theorem 13.1]{ReSi:78})
we have
\begin{align}\nonumber
\lambda_{n}  &= \inf_{X\in\Phi_{n}(\mathcal{D})}\sup_{u\in X,\Vert u\Vert_{L^{2}_{\alpha,c}}=1}\langle P^\#u, u \rangle_{\alpha, c} \\ \nonumber
&=\inf_{X\in\Phi_{n}(\mathcal{D})}\sup_{u\in X,\Vert u\Vert_{L^{2}_{\alpha,c}}=1}\Big(\langle A\nabla u,\nabla u\rangle_{L^{2}(\mathbb{T}_d\backslash\mathcal{O})}+\alpha^{-1}\langle A\nabla u,\nabla u\rangle_{L^{2}(\mathcal{O})}\Big)
\end{align}
where  $\langle, \rangle_{\alpha, c}$
 is the scalar product defined implicitly in Lemma \ref{lem:transmission} by \eqref{eq:measure},  $\Vert \cdot \Vert_{L^{2}_{\alpha,c}}$  is the induced norm, $(\lambda_{n})_{n\geq1}$ denotes the ordered eigenvalues of
$P^\#$, $\mathcal{D}$ is the domain of $P^{\#}$ defined by \eqref{eq:domain_transmission},  and $\Phi_{n}(\mathcal{D})$ the set of all $n$-dimensional
subspaces of $\mathcal{D}$. 
By rescaling
the norms, we then have that
\beq\label{eq:minmax}
\lambda_{n}  =\inf_{X\in\Phi_{n}(\mathcal{D})}\sup_{u\in X,\Vert u\Vert_{L^{2}}=1}\Big(\langle A\nabla u,\nabla u\rangle_{L^{2}(\mathbb{T}_d\backslash\mathcal{O})}+c^2\langle A\nabla u,\nabla u\rangle_{L^{2}(\mathcal{O})}\Big). 
\eeq
Observe that
$$
\mathcal{D} \subset  \big\{(u_{1}, u_{2}) \in H^{1}(\mathbb{T}_{n}\backslash\mathcal{O})\oplus H^{1}(\mathcal{O}) \text{ s.t. }u_{1}=u_{2}\text{ on }\partial\mathcal{O}\big\} = H^1(\mathbb{T}_d), 
$$ 
and thus, by \eqref{eq:minmax},
\begin{equation} \label{eq:mMH1}
\lambda_{n}  \geq \inf_{X\in\Phi_{n}(H^1(\mathbb{T}_d))}\sup_{u\in X,\Vert u\Vert_{L^{2}}=1}\Big(\langle A\nabla u,\nabla u\rangle_{L^{2}(\mathbb{T}_d\backslash\mathcal{O})}+c^2\langle A\nabla u,\nabla u\rangle_{L^{2}(\mathcal{O})}\Big).
\end{equation}
Now, note that if $c \geq 1$ we have
\[
\langle A\nabla u,\nabla u\rangle_{L^{2}(\mathbb{T}_{d}\backslash\mathcal{O})}+c^2\langle A\nabla u,\nabla u\rangle_{L^{2}(\mathcal{O})}\geq\langle A\nabla u,\nabla u\rangle_{L^2(\mathbb{T}_d)},
\]
and thus, by (\ref{eq:mMH1}) and the min-max principle on the torus
\[
\lambda_{n}\geq\lambda_{n}(A, \mathbb{T}_{d}),
\]
and the result follows by the Weyl-type upper bound on Lipschitz compact
manifolds. In the same way, if $c\leq1$, then $\lambda_{n}\geq c^2\lambda_{n}(A, \mathbb{T}_{d})$
and the result follows as well.
\end{proof}

\end{appendix}


\begin{thebibliography}{100}

\bibitem{AbSh:16}
{\sc G.~S. Abeynanda and S.~P. Shipman}, {\em {Dynamic resonance in the high-Q
  and near-monochromatic regime}}, in 2016 IEEE International Conference on
  Mathematical Methods in Electromagnetic Theory (MMET), IEEE, 2016,
  pp.~102--107.

\bibitem{AmChDiDjFi:14}
{\sc M.~Amara, S.~Chaudhry, J.~Diaz, R.~Djellouli, and S.~Fiedler}, {\em A
  local wave tracking strategy for efficiently solving mid-and high-frequency
  helmholtz problems}, Computer Methods in Applied Mechanics and Engineering,
  276 (2014), pp.~473--508.

\bibitem{AmDjFa:09}
{\sc M.~Amara, R.~Djellouli, and C.~Farhat}, {\em {Convergence analysis of a
  discontinuous Galerkin method with plane waves and Lagrange multipliers for
  the solution of Helmholtz problems}}, SIAM J. Num. Anal., 47 (2009),
  pp.~1038--1066.

\bibitem{AnMe:77}
{\sc K.~G. Andersson and R.~B. Melrose}, {\em The propagation of singularities
  along gliding rays}, Invent. Math., 41 (1977), pp.~197--232.

\bibitem{BaDaMo:19}
{\sc S.~Balac, M.~Dauge, and Z.~Moitier}, {\em Asymptotic expansions of
  whispering gallery modes in optical micro-cavities}, preprint,  (2019).

\bibitem{BaYu:16}
{\sc G.~Bao and K.~Yun}, {\em Stability for the electromagnetic scattering from
  large cavities}, Archive for Rational Mechanics and Analysis, 220 (2016),
  pp.~1003--1044.

\bibitem{BaYuZh:12}
{\sc G.~Bao, K.~Yun, and Z.~Zhou}, {\em Stability of the scattering from a
  large electromagnetic cavity in two dimensions}, SIAM Journal on Mathematical
  Analysis, 44 (2012), pp.~383--404.

\bibitem{BaChGo:17}
{\sc H.~Barucq, T.~Chaumont-Frelet, and C.~Gout}, {\em Stability analysis of
  heterogeneous {H}elmholtz problems and finite element solution based on
  propagation media approximation}, Math. Comp., 86 (2017), pp.~2129--2157.

\bibitem{BaSpWu:16}
{\sc D.~Baskin, E.~A. Spence, and J.~Wunsch}, {\em {Sharp high-frequency
  estimates for the Helmholtz equation and applications to boundary integral
  equations}}, SIAM Journal on Mathematical Analysis, 48 (2016), pp.~229--267.

\bibitem{BeMe:18}
{\sc M.~Bernkopf and J.~M. Melenk}, {\em {Analysis of the $ hp $-version of a
  first order system least squares method for the Helmholtz equation}}, in
  Advanced Finite Element Methods with Applications: Selected Papers from the
  30th Chemnitz Finite Element Symposium 2017, Springer International
  Publishing, 2019, pp.~57--84.

\bibitem{BeChGrLaLi:11}
{\sc T.~Betcke, S.~N. Chandler-Wilde, I.~G. Graham, S.~Langdon, and
  M.~Lindner}, {\em {Condition number estimates for combined potential boundary
  integral operators in acoustics and their boundary element discretisation}},
  Numer. Methods Partial Differential Eq., 27 (2011), pp.~31--69.

\bibitem{Bo:01}
{\sc J.-F. Bony}, {\em R\'{e}sonances dans des domaines de taille {$h$}},
  Internat. Math. Res. Notices,  (2001), pp.~817--847.

\bibitem{BoBuRa:10}
{\sc J.-F. Bony, N.~Burq, and T.~Ramond}, {\em Minoration de la r{\'e}solvante
  dans le cas captif}, Comptes Rendus Mathematique, 348 (2010), pp.~1279--1282.

\bibitem{BrGaPe:15}
{\sc D.~L. Brown, D.~Gallistl, and D.~Peterseim}, {\em {Multiscale
  Petrov-Galerkin method for high-frequency heterogeneous Helmholtz
  equations}}, in Meshfree Methods for Partial Differential Equations VIII,
  Springer, 2017, pp.~85--115.

\bibitem{BuNeOk:18}
{\sc E.~Burman, M.~Nechita, and L.~Oksanen}, {\em {Unique continuation for the
  Helmholtz equation using stabilized finite element methods}}, J. Math. Pure.
  Appl.,  (2018).

\bibitem{Bu:98}
{\sc N.~Burq}, {\em D\'ecroissance de l\'energie locale de l'\'equation des
  ondes pour le probl\`{e}me ext\'erieur et absence de r\'esonance au voisinage
  du r\'eel}, Acta Math., 180 (1998), pp.~1--29.

\bibitem{Bu:02a}
\leavevmode\vrule height 2pt depth -1.6pt width 23pt, {\em Lower bounds for
  shape resonances widths of long range {S}chr{\"o}dinger operators}, American
  Journal of Mathematics, 124 (2002), pp.~677--735.

\bibitem{Bu:04}
\leavevmode\vrule height 2pt depth -1.6pt width 23pt, {\em Smoothing effect for
  {S}chr\"{o}dinger boundary value problems}, Duke Math. J., 123 (2004),
  pp.~403--427.

\bibitem{Ca:12}
{\sc Y.~Capdeboscq}, {\em On the scattered field generated by a ball
  inhomogeneity of constant index}, Asymptot. Anal., 77 (2012), pp.~197--246.

\bibitem{CaLePa:12}
{\sc Y.~Capdeboscq, G.~Leadbetter, and A.~Parker}, {\em On the scattered field
  generated by a ball inhomogeneity of constant index in dimension three}, in
  Multi-scale and high-contrast {PDE}: from modelling, to mathematical
  analysis, to inversion, vol.~577 of Contemp. Math., Amer. Math. Soc.,
  Providence, RI, 2012, pp.~61--80.

\bibitem{CaPo:02}
{\sc F.~Cardoso and G.~Popov}, {\em Quasimodes with exponentially small errors
  associated with elliptic periodic rays}, Asymptotic Analysis, 30 (2002),
  pp.~217--247.

\bibitem{CaPoVo:99}
{\sc F.~Cardoso, G.~Popov, and G.~Vodev}, {\em Distribution of resonances and
  local energy decay in the transmission problem {II}}, Mathematical Research
  Letters, 6 (1999), pp.~377--396.

\bibitem{CaPoVo:01}
\leavevmode\vrule height 2pt depth -1.6pt width 23pt, {\em Asymptotics of the
  number of resonances in the transmission problem}, Communications in Partial
  Differential Equations, 26 (2001), pp.~1811--1859.

\bibitem{CaVo:02}
{\sc F.~Cardoso and G.~Vodev}, {\em {Uniform estimates of the resolvent of the
  Laplace-Beltrami operator on infinite volume Riemannian manifolds. II}},
  Annales Henri Poincar{\'e}, 3 (2002), pp.~673--691.

\bibitem{ChGrLaSp:12}
{\sc S.~N. Chandler-Wilde, I.~G. Graham, S.~Langdon, and E.~A. Spence}, {\em
  Numerical-asymptotic boundary integral methods in high-frequency acoustic
  scattering}, Acta Numerica, 21 (2012), pp.~89--305.

\bibitem{ChHeMo:15}
{\sc S.~N. Chandler-Wilde, D.~P. Hewett, and A.~Moiola}, {\em Interpolation of
  {H}ilbert and {S}obolev spaces: quantitative estimates and counterexamples},
  Mathematika, 61 (2015), pp.~414--443.

\bibitem{ChMo:08}
{\sc S.~N. Chandler-Wilde and P.~Monk}, {\em {Wave-number-explicit bounds in
  time-harmonic scattering}}, SIAM J. Math. Anal., 39 (2008), pp.~1428--1455.

\bibitem{ChSpGiSm:17}
{\sc S.~N. Chandler-Wilde, E.~A. Spence, A.~Gibbs, and V.~P. Smyshlyaev}, {\em
  {High-frequency bounds for the Helmholtz equation under parabolic trapping
  and applications in numerical analysis}}, SIAM J. Math. Anal., 52 (2020),
  pp.~845--893.

\bibitem{Ch:15}
{\sc T.~Chaumont~Frelet}, {\em Approximation par {\'e}l{\'e}ments finis de
  probl{\`e}mes d'Helmholtz pour la propagation d'ondes sismiques}, PhD thesis,
  Rouen, INSA, 2015.

\bibitem{ChNi:18}
{\sc T.~Chaumont-Frelet and S.~Nicaise}, {\em {High-frequency behaviour of
  corner singularities in Helmholtz problems}}, ESAIM-Math. Model. Num., 52
  (2018), pp.~1803 -- 1845.

\bibitem{ChNi:18a}
\leavevmode\vrule height 2pt depth -1.6pt width 23pt, {\em Wavenumber explicit
  convergence analysis for finite element discretizations of general wave
  propagation problem}, IMA J. Num. Anal.,
  \url{https://doi.org/10.1093/imanum/drz020} (2019).

\bibitem{ChVa:18}
{\sc T.~Chaumont-Frelet and F.~Valentin}, {\em {A multiscale hybrid-mixed
  method for the Helmholtz equation in heterogeneous domains}}, SIAM J. Num.
  Anal., to appear (2020).

\bibitem{ChLuXu:13}
{\sc H.~Chen, P.~Lu, and X.~Xu}, {\em {A hybridizable discontinuous Galerkin
  method for the Helmholtz equation with high wave number}}, SIAM J. Num.
  Anal., 51 (2013), pp.~2166--2188.

\bibitem{ChQi:17}
{\sc H.~Chen and W.~Qiu}, {\em {A first order system least squares method for
  the Helmholtz equation}}, Journal of Computational and Applied Mathematics,
  309 (2017), pp.~145--162.

\bibitem{ChWu:13}
{\sc H.~Christianson and J.~Wunsch}, {\em {Local smoothing for the
  Schr{\"o}dinger equation with a prescribed loss}}, American Journal of
  Mathematics, 135 (2013), pp.~1601--1632.

\bibitem{CuZh:13}
{\sc J.~Cui and W.~Zhang}, {\em {An analysis of HDG methods for the Helmholtz
  equation}}, IMA Journal of Numerical Analysis, 34 (2013), pp.~279--295.

\bibitem{CuFe:06}
{\sc P.~Cummings and X.~Feng}, {\em {Sharp regularity coefficient estimates for
  complex-valued acoustic and elastic {H}elmholtz equations}}, Math. Mod. Meth.
  Appl. S., 16 (2006), pp.~139--160.

\bibitem{DaDaLa:13}
{\sc M.~Darbas, E.~Darrigrand, and Y.~Lafranche}, {\em {Combining analytic
  preconditioner and fast multipole method for the 3-D Helmholtz equation}}, J.
  Comp. Phys., 236 (2013), pp.~289--316.

\bibitem{DaVa:13}
{\sc K.~Datchev and A.~Vasy}, {\em Propagation through trapped sets and
  semiclassical resolvent estimates}, in Microlocal Methods in Mathematical
  Physics and Global Analysis, Springer, 2013, pp.~7--10.

\bibitem{DeGoMuZi:12}
{\sc L.~Demkowicz, J.~Gopalakrishnan, I.~Muga, and J.~Zitelli}, {\em
  {Wavenumber explicit analysis of a DPG method for the multidimensional
  Helmholtz equation}}, Computer Methods in Applied Mechanics and Engineering,
  213 (2012), pp.~126--138.

\bibitem{DuWu:15}
{\sc Y.~Du and H.~Wu}, {\em {Preasymptotic error analysis of higher order FEM
  and CIP-FEM for Helmholtz equation with high wave number}}, SIAM J. Num.
  Anal., 53 (2015), pp.~782--804.

\bibitem{DuZh:16}
{\sc Y.~Du and L.~Zhu}, {\em {Preasymptotic error analysis of high order
  interior penalty discontinuous Galerkin methods for the Helmholtz equation
  with high wave number}}, Journal of Scientific Computing, 67 (2016),
  pp.~130--152.

\bibitem{DyZw:19}
{\sc S.~Dyatlov and M.~Zworski}, {\em Mathematical theory of scattering
  resonances}, American Mathematical Society, 2019.

\bibitem{EpGrHa:16}
{\sc C.~L. Epstein, L.~Greengard, and T.~Hagstrom}, {\em On the stability of
  time-domain integral equations for acoustic wave propagation}, Discrete Cont.
  Dyn.-A, 36 (2016), pp.~4367--4382.

\bibitem{EsMe:14}
{\sc S.~Esterhazy and J.~Melenk}, {\em {An analysis of discretizations of the
  Helmholtz equation in L2 and in negative norms}}, Comp. Math. Appl., 67
  (2014), pp.~830--853.

\bibitem{EsMe:12}
{\sc S.~Esterhazy and J.~M. Melenk}, {\em On stability of discretizations of
  the {H}elmholtz equation}, in Numerical Analysis of Multiscale Problems,
  I.~G. Graham, T.~Y. Hou, O.~Lakkis, and R.~Scheichl, eds., Springer, 2012,
  pp.~285--324.

\bibitem{FeWu:09}
{\sc X.~Feng and H.~Wu}, {\em Discontinuous {Galerkin} methods for the
  {Helmholtz} equation with large wave number}, SIAM J. Num. Anal., 47 (2009),
  pp.~2872--2896.

\bibitem{FeWu:11}
\leavevmode\vrule height 2pt depth -1.6pt width 23pt, {\em {$hp$-Discontinuous
  Galerkin methods for the Helmholtz equation with large wave number}},
  Mathematics of Computation, 80 (2011), pp.~1997--2024.

\bibitem{FeXi:13}
{\sc X.~Feng and Y.~Xing}, {\em {Absolutely stable local discontinuous Galerkin
  methods for the Helmholtz equation with large wave number}}, Math. Comp., 82
  (2013), pp.~1269--1296.

\bibitem{GaMuSp:17}
{\sc J.~Galkowski, E.~H. M\"{u}ller, and E.~A. Spence}, {\em
  {Wavenumber-explicit analysis for the {H}elmholtz $h$-BEM: error estimates
  and iteration counts for the Dirichlet problem}}, Numerische Mathematik, 142
  (2019), pp.~329--357.

\bibitem{GaSpWu:18}
{\sc J.~Galkowski, E.~A. Spence, and J.~Wunsch}, {\em {Optimal constants in
  nontrapping resolvent estimates}}, Pure and Applied Analysis, 2 (2020),
  pp.~157--202.

\bibitem{GaPe:15}
{\sc D.~Gallistl and D.~Peterseim}, {\em {Stable multiscale Petrov--Galerkin
  finite element method for high frequency acoustic scattering}}, Comput.
  Method. Appl. M., 295 (2015), pp.~1--17.

\bibitem{GiChLaMo:19}
{\sc A.~Gibbs, S.~Chandler-Wilde, S.~Langdon, and A.~Moiola}, {\em A high
  frequency boundary element method for scattering by a class of multiple
  obstacles}, arXiv preprint arXiv:1903.04449,  (2019).

\bibitem{GrLoMeSp:15}
{\sc I.~G. Graham, M.~L\"{o}hndorf, J.~M. Melenk, and E.~A. Spence}, {\em {When
  is the error in the $h$-BEM for solving the {H}elmholtz equation bounded
  independently of $k$?}}, BIT Numer. Math., 55 (2015), pp.~171--214.

\bibitem{GrPeSp:18}
{\sc I.~G. Graham, O.~R. Pembery, and E.~A. Spence}, {\em {The Helmholtz
  equation in heterogeneous media: A priori bounds, well-posedness, and
  resonances}}, Journal of Differential Equations, 266 (2019), pp.~2869--2923.

\bibitem{GrSa:18}
{\sc I.~G. Graham and S.~A. Sauter}, {\em {Stability and finite element error
  analysis for the Helmholtz equation with variable coefficients}}, Mathematics
  of Computation, 89 (2020), pp.~105--138.

\bibitem{HaHu:08}
{\sc H.~Han and Z.~Huang}, {\em {A tailored finite point method for the
  Helmholtz equation with high wave numbers in heterogeneous medium}}, J. Comp.
  Math.,  (2008), pp.~728--739.

\bibitem{He:07}
{\sc U.~Hetmaniuk}, {\em Stability estimates for a class of {H}elmholtz
  problems}, Commun. Math. Sci, 5 (2007), pp.~665--678.

\bibitem{HiMa:12}
{\sc L.~Hillairet and J.~Marzuola}, {\em Nonconcentration in partially
  rectangular billiards}, Analysis \& PDE, 5 (2012), pp.~831--854.

\bibitem{HiMoPe:11}
{\sc R.~Hiptmair, A.~Moiola, and I.~Perugia}, {\em Plane wave discontinuous
  {G}alerkin methods for the 2{D} {H}elmholtz equation: analysis of the
  $p$-version}, SIAM J. Numer. Anal., 49 (2011), pp.~264--284.

\bibitem{HiMoPe:14}
\leavevmode\vrule height 2pt depth -1.6pt width 23pt, {\em {Trefftz
  discontinuous Galerkin methods for acoustic scattering on locally refined
  meshes}}, Applied Numerical Mathematics, 79 (2014), pp.~79--91.

\bibitem{HiMoPe:16}
\leavevmode\vrule height 2pt depth -1.6pt width 23pt, {\em Plane wave
  discontinuous {G}alerkin methods: Exponential convergence of the
  $hp$-version}, Foundations of computational mathematics, 16 (2016),
  pp.~637--675.

\bibitem{HoSh:13}
{\sc R.~Hoppe and N.~Sharma}, {\em {Convergence analysis of an adaptive
  interior penalty discontinuous Galerkin method for the Helmholtz equation}},
  IMA J. Num. Anal., 33 (2013), pp.~898--921.

\bibitem{HuSo:19}
{\sc Q.~Hu and R.~Song}, {\em {A novel least squares method for Helmholtz
  equations with large wave numbers}}, arXiv preprint arXiv:1902.01166,
  (2019).

\bibitem{HuOp:18}
{\sc D.~Huybrechs and P.~Opsomer}, {\em {High-Frequency Asymptotic Compression
  of Dense BEM Matrices for General Geometries Without Ray Tracing}}, J. Sci.
  Comp.,  (2018), pp.~1--36.

\bibitem{Ik:88}
{\sc M.~Ikawa}, {\em Decay of solutions of the wave equation in the exterior of
  several convex bodies}, Ann. Inst. Fourier (Grenoble), 38 (1988),
  pp.~113--146.

\bibitem{Iv:00}
{\sc V.~Ivrii}, {\em Sharp spectral asymptotics for operators with irregular
  coefficients}, International Mathematics Research Notices, 2000 (2000),
  pp.~1155--1166.

\bibitem{Iv:03}
\leavevmode\vrule height 2pt depth -1.6pt width 23pt, {\em Sharp spectral
  asymptotics for operators with irregular coefficients. ii. domains with
  boundaries and degenerations}, Commun. Part. Diff. Eq., 28 (2003),
  pp.~103--128.

\bibitem{Iv:13}
\leavevmode\vrule height 2pt depth -1.6pt width 23pt, {\em Microlocal analysis
  and precise spectral asymptotics}, Springer, 2013.

\bibitem{LaSpWu:19a}
{\sc D.~Lafontaine, E.~Spence, and J.~Wunsch}, {\em {A sharp relative-error
  bound for the Helmholtz $h$-FEM at high frequency}}, arXiv preprint
  arXiv:1911.11093,  (2019).

\bibitem{LaVa:11}
{\sc E.~Lakshtanov and B.~Vainberg}, {\em A priori estimates for high frequency
  scattering by obstacles of arbitrary shape}, Commun. Part. Diff. Eq., 37
  (2012), pp.~1789--1804.

\bibitem{LaLeYo:92}
{\sc C.~Lam, P.~Leung, and K.~Young}, {\em Explicit asymptotic formulas for the
  positions, widths, and strengths of resonances in {M}ie scattering}, Journal
  of the Optical Society of America B, 9 (1992), pp.~1585--1592.

\bibitem{LaPh:89}
{\sc P.~D. Lax and R.~S. Phillips}, {\em Scattering Theory}, Academic Press,
  Boston, 2nd~ed., 1989.

\bibitem{Ph:81}
{\sc P.~T. L\d{a}i}, {\em {Meilleures estimations asymptotiques des restes de
  la fonction spectrale et des valeurs propres relatifs au Laplacien}}, Math.
  Scand., 48 (1981), pp.~5--38.

\bibitem{LiMaSu:13}
{\sc H.~Li, H.~Ma, and W.~Sun}, {\em {Legendre spectral Galerkin method for
  electromagnetic scattering from large cavities}}, SIAM J. Num. Anal., 51
  (2013), pp.~353--376.

\bibitem{LoMe:11}
{\sc M.~L\"{o}hndorf and J.~M. Melenk}, {\em {Wavenumber-Explicit $hp$-{BEM}
  for High Frequency Scattering}}, {SIAM Journal on Numerical Analysis}, 49
  (2011), pp.~2340--2363.

\bibitem{MaIhBa:96}
{\sc C.~H. Makridakis, F.~Ihlenburg, and I.~Babu{\v s}ka}, {\em Analysis and
  finite element methods for a fluid-solid interaction problem in one
  dimension}, Math. Mod. Meth. Appl. S., 6 (1996), pp.~1119--1141.

\bibitem{Mc:00}
{\sc W.~McLean}, {\em {Strongly elliptic systems and boundary integral
  equations}}, Cambridge University Press, 2000.

\bibitem{Me:95}
{\sc J.~M. Melenk}, {\em {On generalized finite element methods}}, PhD thesis,
  The University of Maryland, 1995.

\bibitem{Me:12}
\leavevmode\vrule height 2pt depth -1.6pt width 23pt, {\em {Mapping properties
  of combined field {H}elmholtz boundary integral operators}}, {SIAM Journal on
  Mathematical Analysis}, 44 (2012), pp.~2599--2636.

\bibitem{MePaSa:13}
{\sc J.~M. Melenk, A.~Parsania, and S.~Sauter}, {\em {General DG-methods for
  highly indefinite Helmholtz problems}}, Journal of Scientific Computing, 57
  (2013), pp.~536--581.

\bibitem{MeSa:10}
{\sc J.~M. Melenk and S.~Sauter}, {\em Convergence analysis for finite element
  discretizations of the {H}elmholtz equation with {D}irichlet-to-{N}eumann
  boundary conditions}, Math. Comp, 79 (2010), pp.~1871--1914.

\bibitem{MeSa:11}
\leavevmode\vrule height 2pt depth -1.6pt width 23pt, {\em Wavenumber explicit
  convergence analysis for {G}alerkin discretizations of the {H}elmholtz
  equation}, SIAM J. Numer. Anal., 49 (2011), pp.~1210--1243.

\bibitem{Me:75}
{\sc R.~B. Melrose}, {\em Microlocal parametrices for diffractive boundary
  value problems}, Duke Mathematical Journal, 42 (1975), pp.~605--635.

\bibitem{Me:79}
\leavevmode\vrule height 2pt depth -1.6pt width 23pt, {\em Singularities and
  energy decay in acoustical scattering}, Duke Math. J., 46 (1979), pp.~43--59.

\bibitem{MeSj:78}
{\sc R.~B. Melrose and J.~Sj{\"o}strand}, {\em {Singularities of boundary value
  problems. I}}, Comm. Pure Appl. Math., 31 (1978), pp.~593--617.

\bibitem{MeSj:82}
\leavevmode\vrule height 2pt depth -1.6pt width 23pt, {\em {Singularities of
  boundary value problems. II}}, Comm. Pure Appl. Math., 35 (1982),
  pp.~129--168.

\bibitem{MoSp:19}
{\sc A.~Moiola and E.~A. Spence}, {\em Acoustic transmission problems:
  wavenumber-explicit bounds and resonance-free regions}, Math. Mod. Meth.
  Appl. S., 29 (2019), pp.~317--354.

\bibitem{MuWaYe:14}
{\sc L.~Mu, J.~Wang, and X.~Ye}, {\em {A new weak Galerkin finite element
  method for the Helmholtz equation}}, IMA Journal of Numerical Analysis, 35
  (2014), pp.~1228--1255.

\bibitem{Ne:01}
{\sc J.~C. N{\'e}d{\'e}lec}, {\em {Acoustic and electromagnetic equations:
  integral representations for harmonic problems}}, Springer Verlag, 2001.

\bibitem{Di:19}
{\sc {NIST}}, {\em {Digital Library of Mathematical Functions}}.
\newblock Digital Library of Mathematical Functions,
  \url{http://dlmf.nist.gov/}, 2019.

\bibitem{No:97}
{\sc J.~R. Norris}, {\em {Heat kernel asymptotics and the distance function in
  Lipschitz Riemannian manifolds}}, Acta Mathematica, 179 (1997), pp.~79--103.

\bibitem{OhVe:18}
{\sc M.~Ohlberger and B.~Verf\"{u}rth}, {\em A new heterogeneous multiscale
  method for the {H}elmholtz equation with high contrast}, Multiscale Model.
  Sim., 16 (2018), pp.~385--411.

\bibitem{Pe:17}
{\sc D.~Peterseim}, {\em {Eliminating the pollution effect in Helmholtz
  problems by local subscale correction}}, Mathematics of Computation, 86
  (2017), pp.~1005--1036.

\bibitem{PeZw:99}
{\sc V.~Petkov and M.~Zworski}, {\em Breit-{W}igner approximation and the
  distribution of resonances}, Comm. Math. Phys., 204 (1999), pp.~329--351.

\bibitem{PoVo:99}
{\sc G.~Popov and G.~Vodev}, {\em Resonances near the real axis for transparent
  obstacles}, Communications in Mathematical Physics, 207 (1999), pp.~411--438.

\bibitem{Po:91}
{\sc G.~S. Popov}, {\em {Quasimodes for the Laplace operator and glancing
  hypersurfaces}}, in Microlocal Analysis and Nonlinear Waves, Springer, 1991,
  pp.~167--178.

\bibitem{Ra:79}
{\sc J.~Ralston}, {\em Note on the decay of acoustic waves}, Duke Math. J., 46
  (1979), pp.~799--804.

\bibitem{ReSi:78}
{\sc M.~Reed and B.~Simon}, {\em {Methods of Modern Mathematical Physics Volume
  4: Analysis of Operators}}, Academic Press, 1978.

\bibitem{SaTo:18}
{\sc S.~Sauter and C.~Torres}, {\em Stability estimate for the {H}elmholtz
  equation with rapidly jumping coefficients}, Zeitschrift f\"{u}r Angewandte
  Mathematik und Physik, 69 (2018), p.~69:139.

\bibitem{SaZe:15}
{\sc S.~Sauter and J.~Zech}, {\em {A posteriori error estimation of $hp$-dG
  finite element methods for highly indefinite Helmholtz problems}}, SIAM J.
  Num. Anal., 53 (2015), pp.~2414--2440.

\bibitem{Se:80}
{\sc R.~Seeley}, {\em An estimate near the boundary for the spectral function
  of the {L}aplace operator}, Amer. J. Math., 102 (1980), pp.~869--902.

\bibitem{ShWa:05}
{\sc J.~Shen and L.-L. Wang}, {\em {Spectral approximation of the Helmholtz
  equation with high wave numbers}}, SIAM Journal on Numerical Analysis, 43
  (2005), pp.~623--644.

\bibitem{ShWa:07}
\leavevmode\vrule height 2pt depth -1.6pt width 23pt, {\em {Analysis of a
  spectral-Galerkin approximation to the Helmholtz equation in exterior
  domains}}, SIAM Journal on Numerical Analysis, 45 (2007), pp.~1954--1978.

\bibitem{ShWe:13}
{\sc S.~P. Shipman and A.~T. Welters}, {\em Resonant electromagnetic scattering
  in anisotropic layered media}, Journal of Mathematical Physics, 54 (2013),
  p.~103511.

\bibitem{Sj:97}
{\sc J.~Sj\"ostrand}, {\em A trace formula and review of some estimates for
  resonances}, in Microlocal analysis and spectral theory, L.~Rodino, ed., NATO
  ASI Series, Springer, 1997, pp.~377--437.

\bibitem{SjZw:91}
{\sc J.~Sj\"{o}strand and M.~Zworski}, {\em Complex scaling and the
  distribution of scattering poles}, J. Amer. Math. Soc., 4 (1991),
  pp.~729--769.

\bibitem{SjZw:07}
{\sc J.~Sj\"{o}strand and M.~Zworski}, {\em Fractal upper bounds on the density
  of semiclassical resonances}, Duke Math. J., 137 (2007), pp.~381--459.

\bibitem{Sp:14}
{\sc E.~A. Spence}, {\em Wavenumber-explicit bounds in time-harmonic acoustic
  scattering}, SIAM Journal on Mathematical Analysis, 46 (2014),
  pp.~2987--3024.

\bibitem{Sp:08}
{\sc C.~Spina}, {\em {Kernel Estimates for Markov Semigroups and Parabolic
  Schr\"oodinger Operators}}, PhD thesis, Universit\'a del Salento, 2008.

\bibitem{St:00}
{\sc P.~Stefanov}, {\em Resonances near the real axis imply existence of
  quasimodes}, Comptes Rendus de l'Acad{\'e}mie des Sciences-Series
  I-Mathematics, 330 (2000), pp.~105--108.

\bibitem{St:06}
{\sc P.~Stefanov}, {\em Sharp upper bounds on the number of the scattering
  poles}, Journal of Functional Analysis, 231 (2006), pp.~111--142.

\bibitem{StVo:95}
{\sc P.~Stefanov and G.~Vodev}, {\em Distribution of resonances for the
  {N}eumann problem in linear elasticity outside a strictly convex body}, Duke
  Math. J., 78 (1995), pp.~677--714.

\bibitem{StVo:96}
\leavevmode\vrule height 2pt depth -1.6pt width 23pt, {\em Neumann resonances
  in linear elasticity for an arbitrary body}, Communications in mathematical
  physics, 176 (1996), pp.~645--659.

\bibitem{TaZw:98}
{\sc S.-H. Tang and M.~Zworski}, {\em From quasimodes to resonances}, Math.
  Res. Lett., 5 (1998), pp.~261--272.

\bibitem{TaZw:00}
\leavevmode\vrule height 2pt depth -1.6pt width 23pt, {\em Resonance expansions
  of scattered waves}, Comm. Pure Appl. Math., 53 (2000), pp.~1305--1334.

\bibitem{Ta:76}
{\sc M.~Taylor}, {\em Grazing rays and reflection of singularities of solutions
  to wave equations}, Communications on Pure and Applied Mathematics, 29
  (1976), pp.~1--38.

\bibitem{Va:75}
{\sc B.~R. Vainberg}, {\em {On the short wave asymptotic behaviour of solutions
  of stationary problems and the asymptotic behaviour as $t\rightarrow \infty$
  of solutions of non-stationary problems}}, Russian Mathematical Surveys, 30
  (1975), pp.~1--58.

\bibitem{Vo:92}
{\sc G.~Vodev}, {\em Sharp bounds on the number of scattering poles for
  perturbations of the {L}aplacian}, Comm. Math. Phys., 146 (1992),
  pp.~205--216.

\bibitem{Vo:94}
\leavevmode\vrule height 2pt depth -1.6pt width 23pt, {\em Sharp bounds on the
  number of scattering poles in even-dimensional spaces}, Duke Math. J., 74
  (1994), pp.~1--17.

\bibitem{Wu:13}
{\sc H.~Wu}, {\em {Pre-asymptotic error analysis of CIP-FEM and FEM for the
  Helmholtz equation with high wave number. Part I: linear version}}, IMA J.
  Num. Anal., 34 (2014), pp.~1266--1288.

\bibitem{ZhBa:15}
{\sc L.~Zhao and A.~Barnett}, {\em {Robust and efficient solution of the drum
  problem via Nystrom approximation of the Fredholm determinant}}, SIAM J. Num.
  Anal., 53 (2015), pp.~1984--2007.

\bibitem{ZhDu:15}
{\sc L.~Zhu and Y.~Du}, {\em {Pre-asymptotic error analysis of $hp$-interior
  penalty discontinuous Galerkin methods for the Helmholtz equation with large
  wave number}}, Computers \& Mathematics with Applications, 70 (2015),
  pp.~917--933.

\bibitem{ZhWu:13}
{\sc L.~Zhu and H.~Wu}, {\em {Preasymptotic error analysis of CIP-FEM and FEM
  for Helmholtz equation with high wave number. Part II: hp version}}, SIAM J.
  Num. Anal., 51 (2013), pp.~1828--1852.

\bibitem{Zw:17}
{\sc M.~Zworski}, {\em Mathematical study of scattering resonances}, B. Math.
  Sci., 7 (2017), pp.~1--85.

\end{thebibliography}

\end{document}